\documentclass[11pt,reqno]{amsart}
\usepackage[left=3cm,right=3cm,top=3cm,bottom=3cm]{geometry}
\usepackage[english]{babel}
\usepackage[utf8x]{inputenc}
\usepackage{amsmath}
\usepackage{amsthm}
\usepackage{graphicx}
\usepackage{float}
\usepackage[colorinlistoftodos]{todonotes}
\usepackage{amsmath,amssymb,mathrsfs,amsthm,dsfont}
\usepackage[T1]{fontenc}
\usepackage{tikz}
\usetikzlibrary{matrix}
\title[Moran model with random switching]{Persistence in the Moran model with random switching }

\author[A. Guillin]{\textbf{ {Arnaud} Guillin $^{\diamondsuit}$ }}
\address{{\bf {Arnaud} GUILLIN}\\ Laboratoire de Math\'ematiques Blaise Pascal, CNRS UMR 6620, Universit\'e Clermont-Auvergne,
avenue des Landais, F-63177 Aubi\`ere.} \email{arnaud.guillin@uca.fr}

\author[A. Personne]{\textbf{\quad {Arnaud} Personne $^{\clubsuit}$ }}
\address{{\bf {Arnaud} PERSONNE}\\ CMLA, CNRS UMR 8536, ENS Paris-Saclay,
61 Avenue du Président Wilson, 94230 Cachan.} \email{arnaud.personne@uca.fr}

\author[E. Strickler]{\textbf{\quad {Edouard} STRICKLER $^{\spadesuit}$}}
\address{{\bf {Edouard} STRICKLER}\\ Université de Lorraine, CNRS, Inria, IECL, UMR 7502, F-54000 Nancy, France} \email{edouard.strickler@inria.fr}

\theoremstyle{definition}
\newtheorem{definition}{Definition}
\newtheorem{remarque}{Remark}

\newtheorem{ex}{Example}

\newtheorem{hyp}{Assumption}

\theoremstyle{plain}
\newtheorem{theorem}{Theorem}
\newtheorem{lemma}[theorem]{Lemma}
\newtheorem{prop}[theorem]{Proposition}
\newtheorem{cor}[theorem]{Corollary}
\newtheorem{conj}[theorem]{Conjecture}

\newcommand{\pp}{\mathbb{P}}

\begin{document}

\maketitle

 \begin{center}

\textsc{$^{\diamondsuit}$ Universit\'e Clermont-Auvergne}
\smallskip

\textsc{$^{\clubsuit}$ Ecole normale superieure Paris-Saclay}
\smallskip

\textsc{$^{\spadesuit}$  Université de Lorraine, CNRS, Inria, IECL, UMR 7502}
\smallskip

\end{center}

\begin{abstract}
The paper is devoted to the study of the asymptotic behaviour of Moran process in random environment, say random selection. In finite population, the Moran process may be degenerate in finite time, thus we will study its limiting process in large population which is a Piecewise Deterministic Markov Process, when the random selection is a Markov jump process. We will then study its long time behaviour via the stochastic persistence theory of Benaïm \cite{benaimpersistence}. It will enable us to show that persistence can occur, i.e. asymptotic coexistence of all species, when there are enough switching possibilities. This is true  even if one species has never a predominant selection. \end{abstract}

\bigskip

\bigskip

 
\tableofcontents

%
%
%
%

\section*{Introduction}

Population dynamics is a complex phenomenon in which environment plays a determining role. In particular, environmental fluctuations may be a determining factor to conserve  biodiversity. 
Some works on this subject, such as \cite{kalyuzhny}, \cite{jabotlohier}, have already highlighted its influence on a population evolution and its primary role in the stability of ecosystems. From a theoretical point of view, it has also been proved that environmental fluctuations may permit coexistence in models where a constant environment does not ( \cite{ARMSTRONG1976317}, \cite{deMottoni1981}, \cite{Chesson1982}, \cite{danino2018stability} and  \cite{Chesson}).\\

 Among all  these mathematical models able to include environmental variations, one of the most known in ecology but also in  population genetics is the Moran model \cite{Moran} with selection and immigration. It is a birth and death process: in a given population, an individual is chosen to die uniformly in the population and then a child chooses his parent proportionally to the abundance in the previous population.
 Environment influences birth (and eventually death), each species has a fitness that gives individuals of the same species more or less probabilities of being chosen as parents compared to the neutral case. The fitness changes randomly through
time, and are modeled by a Markov chain.\\

 Immigration represents interactions between the community studied and the external environment, its main role is to introduce new species or reintroduce a species and so to conserve biodiversity. One can also consider mutation, but in fact it has the same impact on the process, as we are in a fixed population size. Without immigration (or mutation) and for a fixed population size, it well known that a species invade definitively the community in a finite time \cite{coalescentselection}. But what happen for large population? Our objective is to prove that, in the large population limit  (i.e the population size goes to infinity), environmental fluctuations can be sufficient to conserve biodiversity in the Moran model. It will be done through the notion of stochastic persistence developed by Benaïm and Schreiber in \cite{benaimschreiber} and in  \cite{benaimpersistence}.\\\medskip
 
 Let us present the structure of the paper and the main results we obtain. In a first section, we present the Moran Model with random selection and we obtain a quantitative approximation of this discrete process (rescaled in time with respect to  the size of the population) in the limit in large population towards a particular Piecewise Deterministic Markov Process (PDMP), which is the content of Section 1 and particularly Theorem 2. We thus have to study this PDMP in which only the selection parameter is random and is a pure jump Markov process. For two species the PDMP is the following (we only need to follow the dynamic of one species):

\begin{equation}  
\label{2species}
\left\{
\begin{aligned}
&\frac{d X_t}{dt} = s_{t}\frac{X_t(1-X_t)}{1+s_{t}X_t}\\
&\mathbb{P}_{x,s}(s_{t+h}=s_k|s_t=s_j)=q_{j,k}h+o(h) \quad \mbox{if} \quad j \neq k
\end{aligned}
\right. 
\end{equation}
where $s_t$ is our selection jump Markov process taking values in some finite space. 
 
In order to understand the long time behavior of this process, it is first particularly informative to understand this dynamic when the selection is constant, which is done in Section 2.

Then we will consider the long time behavior when the selection is a jump Markov process. To this end we introduce in Section 3 the notion of conservation of biodiversity, taken from Schreiber \cite{Schreiber}, see also \cite{benaimschreiber}. We use the the definition of stochastic persistence \cite{benaimpersistence} which (roughly) asserts that the process spend  (a.s.) an infinite time away from the absorbing boundaries.

Section 4 is dedicated to the case where there are two species and a switch between two different environments. It is sufficient as more possible values of the environment may be captured by this setting, as the important feature is to switch from one favorable environment to a defavorable one. Proposition 1 gives explicit condition to get stochastic persistence or extinction of one of the species. For example, in Example \ref{ex1}, when the jump rates are constant and equal, stochastic persistence is ensured as soon as the two values of the selection process satisfies $-s_1<s_2<-s_1/(1+s_2)$, that we illustrate numerically.

The general case is treated (partially) in Section 5, where are exhibited conditions under which a species may go to extinction or one species invade the other ones. However criterion for stochastic persistence given by \cite{benaimpersistence} are quite difficult to handle in more than 2 species as it requires the explicit knowledge of all the ergodic measures. Therefore we consider and detail the case of 3 species in Section 6. When there are only two different environments, unfortunately there is no stochastic persistence. However some interesting phenomenon takes place such as a specie never favored nor defavored may nevertheless invade the community. Turning to three different environments we exhibit numerical conditions under which stochastic persistence occurs.\\
 An appendix gathers the proof of Theorem 2 and some useful lemma.

\section{Moran model with selection and immigration and its large population limit}
\subsection{Presentation of the model}
\label{presentation}
\quad\\
\label{presentationmodele}
The Moran process was introduced in population genetics in 1958 \cite{Moran} 
to describe the probabilistic dynamics in a  constant size population in which
many alleles compete for dominance. Similarly, it models the 
stochastic dynamics  of a population in which several species coexist.
In this section we describe in details the discrete model, i.e. the Moran process. 
 A particularity of this model is that an event occurs at each time step. More precisely each event corresponds to the death of an individual and the birth of another who replaces it.\\
We consider a population, whose size is constant over time equal to $J$, composed of $S+1$ species . The proportion of the $i$ species at the $n^{th}$ event is denoted $X^{i}_{n}$, $i \in \mathbb{S}=\{1,...,S+1\}$, $n\in \mathbb{N}$.

As usual once we know $(X_n^{i})_{i=1,..,S}$, we deduce the proportion for the last species, $X_{n}^{S+1}=1-\sum\limits_{i=1}^{S}X^{i}_{n}$. We denote by $X_{n} = (X_n^{i})_{i=1,..,S}$ the species vector or abundance vector. The dynamics of evolution follows the following pattern at the step $n$:
\begin{enumerate}
\item The individual designated to die is chosen uniformly among the community.
\item It is replaced by an individual that chooses  its parent randomly in the community. 
The parent is then of the species $i$ with probability
\[\frac{X^{i}_{n}(1+s^{i}_{n})}{1+\sum\limits_{k=1}^{S+1}X_n^{k}s^{k}_{n}}.\] 
\end{enumerate}
The $(s_{n}^{i})_{ i \in \mathbb{S}}$ are the selection parameters (or fitness) which evolve through time. They  can be seen as an advantage (or a disadvantage) giving more weight to the $i^{th}$ species. In a neutral case, where all the species have the same fitness, the above probability is equal to $X^i_n$, meaning that no species is advantaged. 

We denote by $s_n$, the environment at step $n$, i.e the vector having for $i^{th}$ coordinate $s_n^i$. In the following,  $(s_n)_{n>0}$ will be a Markov chain taking values in a finite space $E$ with cardinality $K$. Furthermore, we  assume that for all $n \geq 0$,  $s_{n}^{S+1}=0$. Indeed, we can obtain it from any configuration by changing all the coefficients by  $s_{n}^{i}=\frac{\tilde{s}^{i}_{n}-\tilde{s}^{S+1}_{n}}{1+\tilde{s}_{n}^{S+1}} $. This assumption forces, if we take initially positive fitness, that the coefficients  belong to the set $]-1;\infty[$

We assume throughout this work that $(s_n)_{n \geq 0}$ is autonomous, meaning  that its evolution does not depend on $(X_n)_{n\ge0}$. We further define $U_n$ as the vector composed of $X_n$ and $s_n$. 

This model therefore describes a Markovian dynamic in which a species invades definitively the community almost surely in a finite time. This is due to the absence of immigration. Some works, such as \cite{danino2018fixation} and \cite{danino2018stability},  give an estimation of the absorbing time in non random environment.  However, the invasion time is increasing with the   population size $J$.
Our main goal is  to understand the behaviour of the process when the population is large. Do environmental switch improve or reduce biodiversity?

The species vector $(X_n)_{n \geq 0}$ is a Markovian random process. 
 Let us describe the transition probabilities.
Let $x$ be the vector having for coordinate $i$, $x^{i}$.
Denote $\Delta=\frac{1}{J}$, so for the $i$ species:

    \begin{eqnarray*}
     P_{x^{i}+}&=& \pp\left(X^{i}_{n+1}=x^i+\Delta|U_{n}=(x,s_n)\right)=(1-x^{i})\,\frac{x^{i}(1+s^{i}_{n})}{1+\sum\limits_{k=1}^{S+1}x^{k}s^{k}_{n}},
     \end{eqnarray*}
     \begin{eqnarray*}
  P_{x^{i}-}   &=&\pp\left(X^{i}_{n+1}=x^i-\Delta|U_{n}=(x,s_n)\right)=x^{i}\left(1-\frac{x^{i}(1+s^{i}_{n})}{1+\sum\limits_{k=1}^{S+1}x^{k}s^{k}_{n}}\right),
\end{eqnarray*}
     \begin{eqnarray*}
    P_{x^{i+}x^{j-}} &=&\mathbb{P}(\{X^{i}_{n+1}=x^{i}+\Delta\} \cap\{ X^{j}_{n+1}=x^{j}-\Delta\}|U_{n}=(x,s_n))=x^{j}\left(\frac{x^{i}(1+s_{n}^{i})}{1+\sum\limits_{k=1}^{S+1}x^{k}s_{n}^{k}}\right).
\end{eqnarray*}
When the population size $J$ is big, understanding the temporal evolution of the population is not easy from a mathematical and computing  point of view. Thus we will approach the dynamics of this model by a continuous-time random differential equation when the population is large.

\subsection{In large population}

Let us start with the following proposition that characterises the order (relative to J) of the expectation, variance, and covariance of the abundance variation of a species during an event:
\begin{prop}
\label{expvar2}
\quad\\
Let $\Delta_{t}=\frac{1}{J}$ and $n=tJ$. When $J$  goes to infinity, we have :
\begin{enumerate}
\item $ E[X^{i}_{n+1}-X^{i}_{n}|U_{n}=(x,s_n)]=\frac{1}{J}x^{i}\left(s_n^{i}-\sum\limits_{k=1}^Sx^{k}s_n^{k}\right)(1+\sum\limits_{k=1}^Sx^{k}s_n^{k})^{-1}$,
\item $Var[X_{n+1}-X_{n}|X_{n}=(x,s_n)]=O(\frac{1}{J^2})$,
\item $Cov[X^{i}_{n+1},X^{l}_{n+1}|U_{n}=(x,s_n)]=O(\frac{1}{J^2}) $.

\end{enumerate}
\end{prop}

This proposition shows that the expectation is of the order of $ \frac{1}{J} $ whereas the variance is of order $ \frac{1}{J ^ 2}$. Considering a scale in $\frac{1}{J}$, the variation of variance over a time step becomes negligible in large population limit. Thus, we should obtain a limit process where only $s$ could be random. With this choice, we keep a selection  independent of the size of the population and the limit process is usually called the "strong selection approximation", see \cite{Ethier1989} and \cite{kimuraapproxdiff}. It is opposed to another possible choice, the "weak selection approximation", which consists in considering the selection inversely proportional to the size of the population. The paper \cite{gjp2018} deals with this other possibility.

To emphasise the dependence on $J$, we use in this section the notation $U_n^J=(X_n^J, s_n^J)$.  We give an approximation result under the following assumption. 
\begin{hyp}
\label{hyp:limitMC}
There exists a finite set $E$ such that, for all $J$, the  process $(s_n^J)_{n \geq 0}$  is an autonomous Markov chain defined on $E$. Moreover, if we denote by $P_{s,s'}^J$ the transition probability of $(s_n^J)_{n \geq 0}$, then for all $s\neq s'$, there exist $\alpha_s > 0$ and $Q_{s,s'} \geq 0$ such that 
\[
\lim_{J \to \infty} J P_{s,s'}^J = \alpha_s Q_{s,s'}.
\]
\end{hyp}
We consider the rescaled piecewise linear extension  $\tilde s_t^J=s_{\lfloor tJ\rfloor}^J$ of $(s_n^J)_{n \geq 0}$. Let  denote  by $P_{\tilde{s}^J}(s,s',t)$ the transition probabilities of $\tilde{s}^J$, then  for all $s\not=s'$,$$\lim\limits_{J\rightarrow\infty}P_{\tilde{s}^{J}}\left(s,s',\frac{1}{J}\right)\times J =\alpha_sQ(s,s').$$ 
We also consider the interpolated continuous-time process $(\tilde{X}^J_t)_{t \geq 0}$ of $(X_n^J)_{n \geq 0}$ defined by

\[
\tilde{X}^J_t = X^J_{\lfloor tJ\rfloor} + J \left( t - \frac{\lfloor tJ\rfloor}{J} \right) \left( X^J_{\lfloor tJ\rfloor + 1}- X^J_{\lfloor tJ\rfloor}\right),
\]
and we set $(\tilde{U}^J_t)_{t \geq 0} = (\tilde{X}_t^J, \tilde{s}_t^J)_{t \geq 0}$.
Theorem \ref{thqttapprox} below means, under Assumption  \ref{hyp:limitMC}, when $J$ tends to infinity the sequence of processes $(\tilde{U}^J_t)_{t \geq 0} $ converges towards $(U_t)_{t \geq 0} = (X_{t},s_t)_{t \geq 0} $, where $(X_t)_{t>0}$ is solution of
\begin{equation}
\begin{pmatrix}
dX^{1}_{t}\\
\vdots\\
dX^{S}_{t}
\end{pmatrix}
=
\begin{pmatrix}
X^{1}_{t}\frac{ s_{t}^1-\sum\limits_{k\in \mathbb{S}}X_t^{k}s^{k}_{t}}{1+\sum\limits_{k\in \mathbb{S}}X_t^{k}s^{k}_{t}}   \\
\vdots\\
X^{S}_{t}\frac{ s_{t}^S-\sum\limits_{k\in \mathbb{S}}X_t^{k}s^{k}_{t}}{1+\sum\limits_{k\in \mathbb{S}}X_t^{k}s^{k}_{t}}  
\end{pmatrix}dt
\label{eqprincipalepdmp} 
\end{equation}
and $(s_{t})_{t \geq 0}$ is a continuous times Markov chain with $\alpha Q$ for generator.

The process $(U_t)_{t \geq 0}$ is a  \emph{Piecewise Deterministic Markov Processes} (PDMP). It evolves according to the deterministic dynamics of the equation \eqref{eqprincipalepdmp} for some $ s $  fixed during a random time with exponential distribution. Then its behavior changes and adopts another dynamic when the parameter $ s $ switches. This kind of process was introduced by Davis \cite{Dav84}. The PDMPs' have become ubiquitous in stochastic modelling of various phenomena. They are applied to neuroscience \cite{PTW10}, \cite{PTW102}, \cite{PTW12}, genetics \cite{CDMR12},  ecology \cite{BL16}, internet traffic \cite{CMP10}, \cite{FGM12}, \cite{FGM16}, \cite{BCGMZ13}. See also \cite{azais}, \cite{BMZIHP} and \cite{cloezmalrieu} and the references therein for more details and applications. 

For $(U_t)_{t>0}$ and more generally for Markov processes, we use the following standard notations. If $\mu$ is a probability measure, we let $\mathbb{P}_{\mu}$ denote the law of $(U_t)_{t>0}$ knowing that $U_0$ has distribution $\mu$, and $\mathbb{E}_{\mu}$ denote the associated expectation. In the special case where $\mu = \delta_{(x,s)}$, then we write $\mathbb{P}_{x,s}$ instead of $\mathbb{P}_{\delta_{x,s}}$.

Consider now the process $(Z^J_{t})_{t \geq 0} = (U_{t/J})_{t \geq 0}$, taking values in $\mathscr{I}=[0.1]\times E$.  Its generator acts on $f$ in $\mathscr{C}^1(\mathscr{I})$ as : 
\begin{align*}
 Lf(x,s)=&\frac{sx(1-x)}{(1+sx)J}\frac{{\partial}}{\partial{x}}f(x,s)
 +\sum\limits_{s'\in E}\frac{\alpha_sQ_{s,s'}}{J}\big(f(x,s')-f(x,s)\big)\quad \\
 &=  L_Cf(x,s)+ L_D f(x,s)
 \end{align*}
Recall that our  aim is to compare $(\tilde{U}^J_t)_{t \geq 0}$ with $(U_t)_{t \geq 0}$. We will be interested in quantities of the type 
\[
\mathbb{E}_{x,s}\left[ f\left(\tilde{U}_t^J\right) - f(U_t) \right],
\]
 for $f : \mathscr{I}\to \mathbb{R}$ regular enough and $(x,s) \in I_J   \times E $, where $I_{J}=\lbrace \frac{i}{J}: i=0,1,2,\cdots,J\rbrace $ is the state space of $(X^J_n)_{n \geq 0}$. On the space of continuous functions $f : \mathscr{I} \to \mathbb{R}$, we consider the semi-norm $\Vert f\Vert_J=\max\limits_{(x,s) \in I_J   \times E  }\vert f(x,s)\vert$. For $k \geq 1$, we say that  $f : \mathscr{I} \to \mathbb{R}$ is of class $C^k$ if  for all $s$, the map $f_s : x \mapsto f(x,s)$ is of class $C^k$. In that case, we denote by $f^{(k)}$ the application defined on $\mathscr{I}$ by $f^{(k)}(x,s) = f_s^{(k)}(x)$.
  
  The following theorem gives an estimation of the error of approximation of our initial Moran process by our PDMP.
 
\begin{theorem}
\label{thqttapprox}
There exist a function $q : \mathbb{R}_+ \to \mathbb{R}_+$ at most exponential and a function $k : \mathbb{R}_+ \to \mathbb{R}_+$ linear in time such that for all  $f \in C^{3}(\mathscr{I}) $, for all $t \geq 0$, 
\begin{align*}
\max\limits_{(x,s) \in I_J   \times E  } \mathbb{E}_{x,s}[f(\tilde{U}^J_{t})-f(U_t)]&\le \frac{q_t}{J}\Big(\Vert f\Vert_J+\Vert f^{(1)} \Vert_{J}+\Vert f^{(2)} \Vert_{J}\Big)\\ &+k_t\max\limits_{s,s' \in E}\Big|JP_{s,s'}-\alpha_{s}Q_{s,s'}\Big|\Vert f\Vert_{J}
+O(\frac{1}{J^2})
\end{align*}

\end{theorem}
The sketch of the proof is the same as in the article \cite{ggp2018}. The complete proof is in appendix. To understand the general dynamic of the PDMP, we  first study the vector field of the equation \eqref{eqprincipalepdmp} for a non random environment. The following section is dedicated to this study.

\section{Vector field  for constant parameters.}
\label{dim2champvect}
In all this section, we assume that $s$ is a constant process. In particular, $(X_t)_{t>0}$ is a deterministic process, solution of the ODE 
\begin{equation}
\begin{pmatrix}
\frac{dX^{1}_{t}}{dt}\\
\vdots\\
\frac{dX^{S}_{t}}{dt}
\end{pmatrix}
=
\begin{pmatrix}
X^{1}_{t}\frac{ s^1-\sum\limits_{k\in \mathbb{S}}X_t^{k}s^{k}}{1+\sum\limits_{k\in \mathbb{S}}X_t^{k}s^{k}}   \\
\vdots\\
X^{S}_{t}\frac{ s^S-\sum\limits_{k\in \mathbb{S}}X_t^{k}s^{k}}{1+\sum\limits_{k\in \mathbb{S}}X_t^{k}s^{k}}  
\end{pmatrix}
\label{ODEnoswitch}
\end{equation}
Let us first define the space  where the process $X_t$ evolves: $$\mathscr{E}=\{ X \in \mathbb{R}_+^S \: :  \sum\limits_{i=1}^SX^i\leq 1  \}.$$ Note that $\mathscr{E}$ can also be seen as the unit simplex in dimension $S+1$: $\mathscr{E} \simeq \Delta = \{ X \in \mathbb{R}_+^{S+1} \: :  \sum\limits_{i=1}^{S+1} X^i =1  \}.$
That is why, for  $X \in \mathscr{E}$, we set $X^{S+1} = 1 - \sum\limits_{i=1}^S X^i$. 

For $1 \leq i \leq S+1$, we let $\mathscr{E}^i_0$ be the extinction set of species $i$ : 
\[
\mathscr{E}_0^i=\{X \in \mathscr{E} \: :   X^i=0\}.
\]
 We also denote by  $\mathscr{E}_0$ the extinction set, i.e the set where at least one species is extinct: 
\[
\mathscr{E}_0 = \bigcup_{i=1}^{S+1} \mathscr{E}^i_0.
\]
For $ 1 \leq S$, let $e_i$ denote the $i^{th}$ vector of the standard basis and set $e_{S+1} = 0$. The point $e_i$ correspond to the invasion of species $i$, that is, species $i$ is the only species in the system. Note that when we see the process as a process in $\Delta$, $e_{S+1}$ is the $(S+1)^{th}$ vector of the standard basis.
Finally, we set $\mathscr{E}_+^i = \mathscr{E} \setminus \mathscr{E}_0^i$, the set where species $i$ is not extinct and $\mathscr{E}_+ = \mathscr{E} \setminus \mathscr{E}_0$ the set where  none of  the species is  extinct.

Let us remark that $\mathscr{E}_0$ corresponds to the edges of the set $\mathscr{E}$, and $\mathscr{E}_+$ to the interior of $\mathscr{E}$. It is also easily seen that $\mathscr{E}$, $\mathscr{E}_0^i$ , $\mathscr{E}_0$, $\mathscr{E}_+^i$ and $\mathscr{E}_+$ are invariant sets for $X_t$ i.e $X_t \in \mathscr{E}$ (respectively $\mathscr{E}_0^i$, $\mathscr{E}_0$, $\mathscr{E}_+^i$, $\mathscr{E}_+$) if and only if $X_0 \in \mathscr{E}$ (respectively $\mathscr{E}_0^i$, $\mathscr{E}_0$ , $\mathscr{E}_+^i$, $\mathscr{E}_+$).

The following theorem states that when the environment does not switch, the species with the highest fitness  will invade the community.

\begin{theorem}
\label{th:deterministe}
\quad\\
Let $X_t$ be the solution of \eqref{ODEnoswitch}.
Assume that  $s^k\neq s^j$ for all $j \neq k$ and set $m=\arg\max\limits_{i\leq S+1}s^i$.\\ 
Then if $X_0^m >0$,  $X_t$  converges to $e_m$.
\end{theorem}

\begin{proof}



 By assumption, $s^m > s^j$ for all $j \neq m$. In particular, this implies that $s^m \geq \sum\limits_{j=1}^S s^j X^j$, with strict inequality as soon as $X^m \neq 1$. In particular, $\frac{dX^m_t}{dt} \geq 0$, and thus $X_t^m$ is non-decreasing. Since $X_t^m$ is bounded above by $1$, we conclude that $X_t^m$ converges to some $X^m$ and  $\frac{dX^m_t}{dt} $ goes to $0$. Now assume that $X_0^m > 0$, then since $X_t^m$ is increasing, $X^m > 0$. Thus $\lim_{t \to \infty} \frac{d X_t^m}{dt} = 0$ implies that $\lim_{t \to \infty}  \sum\limits_{j=1}^{S+1} (s^m - s^j) X^j_t = 0$, which due to $s^m > s^j$ implies that $\lim_{t \to \infty} X_t^j= 0$ for all $j \neq m$.  This concludes the proof.
\end{proof}

\begin{remarque}
One can actually prove that $X_t$ converges to $e_m$ exponentially fast, with rate $\Lambda := - \min\limits_{k \neq m} \frac{  (s^m - s^k)}{1 + s^m}.$ For a proof of this result, we refer to theorem \ref{th:alwaysthebest}.
\end{remarque}
\begin{remarque}
If some  fitnesses are the same, i.e. $s^j=s^i$ for some $i\neq j$, then the same proof shows that all the species that do not have the best fitness go to extinction. In other words, for all $i$ such that $s^i < s^m$, species $i$ goes to extinction.
\end{remarque}

Figures \ref{fig:image12} and \ref{fig:trajX} give some examples of vector fields and the trajectories for different initial conditions for one particular vectors field.

\begin{figure}[htbp]
\begin{minipage}[c]{.45\linewidth}
\begin{center}
\includegraphics[scale=0.35]{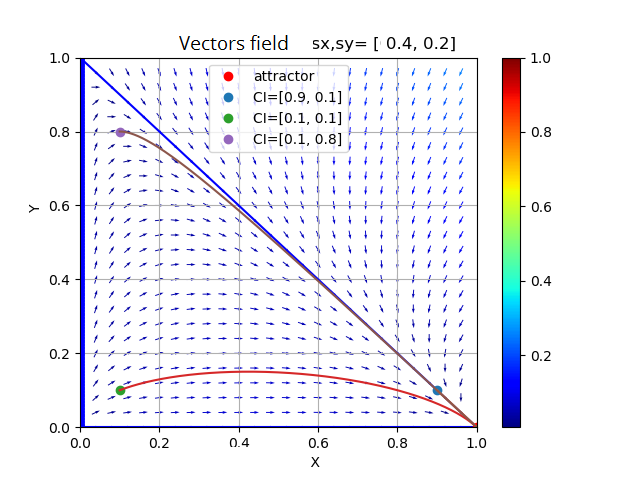}
\end{center}
\end{minipage}
\hfill
\begin{minipage}[c]{.45\linewidth}
\begin{center}
\includegraphics[scale=0.35]{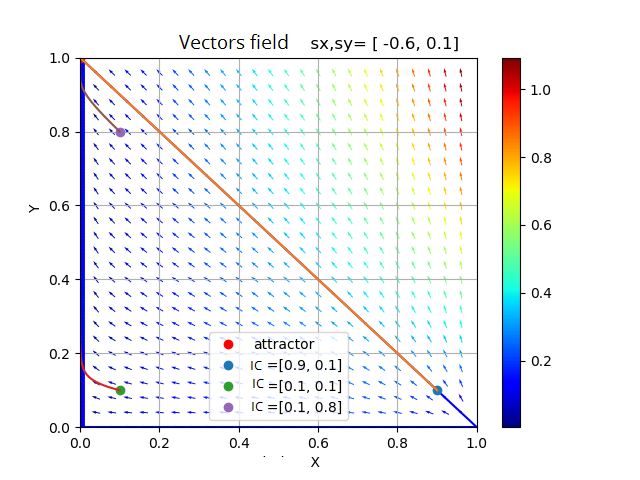}
\end{center}
\end{minipage}
\caption{ Vectors fields and trajectories for different initial conditions (3 species) }.
\label{fig:image12}
\end{figure}

\begin{figure}[htbp]
\begin{minipage}[c]{.45\linewidth}
\begin{center}
\includegraphics[scale=0.35]{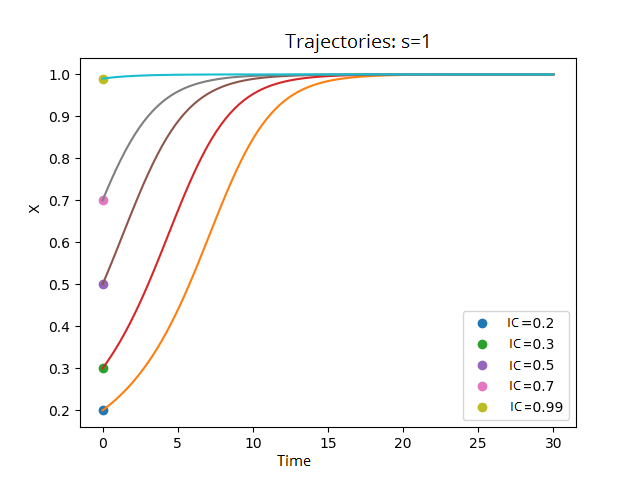}
\end{center}
\end{minipage}
\hfill
\begin{minipage}[c]{.45\linewidth}
\begin{center}
\includegraphics[scale=0.35]{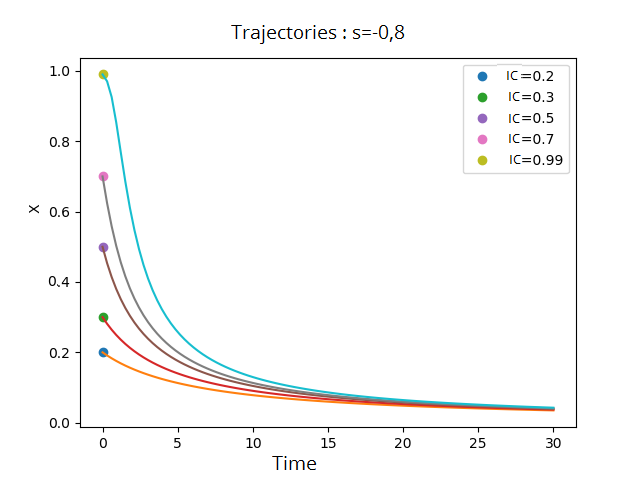}
\end{center}
\end{minipage}
\caption{Trajectories of $X_t$ for different initial conditions (S=2).}
\label{fig:trajX}
\end{figure}

The following sections deal with the random process with switched environments. First  we treat the case of two species, which is basic  but essential to understand the behavior of our population with more species.

\section{General framework and tools}

In this section,  $(s_t)_{t>0}$ is Markovian jump process taking values in  a finite space  $E=\{s_1,...,s_K\}$ and having for generator $(Q_{i,j})_{1\leqslant i,j \leqslant K}$. Let $s_j^i$ be the  selection parameter of the $i^{th}$ species in environment $s_j$.

In the discrete model,  even if $(s_t)_{t>0}$ is a Markov chain, one species invades the community in a finite time almost surely. This property  may still be preserved in weak selection  (see for example \cite{gjp2018}, where examples are given on the selection process leading to a reachable boundary).

In strong selection, even if the selection is deterministic, it is impossible for a species to reach the extinction set in finite time but the process may concentrate in a neighbourhood of extinction and no longer emerge over time, i.e., extinction occurs at an infinite horizon and it is equivalent to a loss of biodiversity.
In the following, we prove that under some assumptions, in strong selection it is possible to conserve biodiversity in the community. Moreover we give some information on the long time behavior of the process. 
 
\subsection{General framework}
In this part, we recall some recent results of Benaïm \cite{benaimpersistence} that will be used in the paper.
With the notations of the previous section, we set $M=\mathscr{E}\times E$, $M_0 = \mathscr{E}_0 \times E$ and $M_+ = M \setminus M_0$. We define in a similar way $M_0^i$ and $M_+^i$.

We consider the process $(U_t)_{t>0}=(X_t,s_t)$ on $M$ starting from $(x,s)$ and defined by :

\begin{equation} \label{process}
\left\{
\begin{aligned}
&\frac{dX_t}{dt}=G_{s_t}(X_t)\\
&\mathbb{P}_{x,s}(s_{t+h}=s_k|s_t=s_j)=q_{j,k}h+o(h) \quad \mbox{if} \quad j \neq k
\end{aligned}
\right. 
\end{equation}
with $G^i_{s}(X)=X^iF^i_{s}(X)$ and $F^i_{s}(X)=\frac{s^i-\sum\limits_{k=1}^Ss^k X^k}{1+\sum\limits_{k=1}^S s^k X^k}$ and $q_{j,k}$ are the generator coefficients of the Markovian jump process.


Finally let $(\Pi_t^u)_{t>0}$ be the empirical occupation measure of the process $(U_t)_{t \geq 0}$, for  $U_0=u=(x,s)$, defined by 
$$\Pi_T^u(B)=\frac{1}{T}\int_0^T \mathds{1}_{\{U_t\in B\}}ds, \quad \forall B \in \mathscr{B}(M)$$
\subsection{Stochastic Persistence}$~$\\
The following definition follows from \cite{Schreiber}.
\begin{definition}
The family $\{(U_t)_{t>0},U_0 \in M^+ \}$ is \textit{stochastically persistent} with respect to $M_0$ if for all $\epsilon>0$ there exists a compact set $K_{\epsilon}\in M^+$ such that for all $u$ in $M^+$: \[\mathbf{P}(\lim_{t\rightarrow 0}\inf\Pi^{u}_t(K_{\epsilon})\geqslant 1-\epsilon)=1\]
\end{definition}
This  definition means that all the species, initially present, stay away from the extinction set over arbitrary long periods of time. Persistence with respect to $M_0^i$ is defined in the same way. 
 
\begin{definition} A probability measure is said to be ergodic for a Markovian process if it is invariant and extremal for the process, meaning that it cannot be written as a nontrivial convex combination of other invariant measures.
 For a Borelian set $B$ of M, we denote by $\mathscr{P}_{erg}(B)$ the space of ergodic probability measure such as for all $\mu$ in $\mathscr{P}_{erg}(B)$ , $\mu(B)=1$.
 
 \end{definition}
We recall the definition given in \cite{benaimpersistence} of the invasion rates with respect to an ergodic probability measure in $\mathscr{P}_{erg}(M_0)$.
\begin{definition}
\label{tauxinvasion}
For $\mu \in \mathscr{P}_{erg}(M_0)$, we introduce $\mu_j(B)=\mu(B\times{s_j})$. \emph{The invasion rate of species $i$ with respect to $\mu$} is defined by :
\[
\lambda_i(\mu)=\sum\limits_{1 \leqslant j \leqslant K }\int_{M_0} F_{s_j}^i(x)d\mu_j(x).\]
\end{definition}

\begin{remarque}
\label{rem:invasionrates}
The intuition behind these quantities is the following. On the one hand, from equation \eqref{process}, we see that whenever $X^i_0 \neq 0$, one has 
\[
\frac{1}{t} \log (X_t^i) = \frac{1}{t} \int_0^t F^i_{s_u}(X_u) du + \frac{1}{t} \log ( X_0^i ).
\]
In particular, 
\[
\frac{1}{t} \log (X_t^i) \sim \frac{1}{t} \int_0^t F^i_{s_u}(X_u) du 
\]
as $t \to \infty$. On the other hand, Birkhoff ergodic's Theorem states that for $\mu \in \mathscr{P}_{erg}(M_0)$, for $\mu$ almost every $u_0$, $\pp_{u_0}$ almost surely, 
\[
\lim_{t \to \infty} \frac{1}{t} \int_0^t F^i_{s_u}(X_u) du  = \sum\limits_{1 \leqslant j \leqslant K }\int_{M_0} F_{s_j}^i(x)d\mu_j(x) = \lambda^i(\mu).
\]
Thus,  $\lambda^i(\mu)$ represent an exponential growth rate of species $i$ near $M_0^i$.
\end{remarque}

The following theorem is a consequence of \cite[Th. 6.1]{benaimpersistence}. It gives sufficient conditions for the process to be persistent with respect to $M_0^i$ and $M_0$, respectively.
\begin{theorem}
\label{thm:persistence}
\begin{enumerate}
\item Assume that for all 
$\mu \in \mathscr{P}_{erg}(M_0^i)$,  $\lambda^i(\mu)>0.$ Then the process given by \eqref{process} is H-persistent  with respect to $M_0^i$.
\item Assume that there exists positive numbers $\{c^i, 1\leqslant i \leqslant S+1\}$ such that, for all $\mu \in \mathscr{P}_{erg}(M_0)$,  $$\sum\limits_{ 1 \leq i \leqslant S+1}c^i\lambda_i(\mu)>0.$$
Then the process given by \eqref{process} is H-persistent  with respect to $M_0$.
\end{enumerate}
\end{theorem}
The first point of the above theorem can be interpreted as follows. If for some species $i$, one has $\lambda_i(\mu) > 0$ for all $\mu \in \mathscr{P}_{erg}(M_0^i)$, then the face $\mathscr{E}_0^i$ is \textit{repulsive} : if the process start in $\mathscr{E}_+^i$ and close to the face $\mathscr{E}_0^i$, then the process is in some sense pushed away from $\mathscr{E}_0^i$. The second point states that the process is pushed away from every face, and thus concentrates on the interior $M_+$ of the domain. 

\subsection{Behavior near a common zero}
The previous theorem gives some way to understand the behavior of the process near the boundary $\mathscr{E}_0^i$, which corresponds to the extinction of species $i$. It is also important to understand the process near the vertex $e_i$, which corresponds to the invasion of species $i$. Since the $\{e_i\}_{i<S+1}$ are common zero of the vector fields $G_s$, we can use the recent results of \cite{BS17} and \cite{S18}. According to these papers, the behavior of the process $(U_t)_{t>0}$ near $e_i$ for some $ 1 \leq i \leq S+1$ is controlled by the behavior of the linear process $V_t = (Y^i_t, s_t)$, where $(s_t)_{t \geq 0}$ is the Markov process with generator $Q$ and $Y^i$ evolves according to 
\[
\frac{d Y^i_t}{dt} = A^i_{s_t} Y^i_t,
\]
where $A_s$ is the Jacobian matrix of $G_s$ at $e_i$. Assume with no loss of generality that $i=1$. Then $A^i_s = DG_s(e_1)$ is given by
\[
A^i_s = \begin{pmatrix}
\frac{-s^1}{1+s^1} & *\\
0 & D_s
\end{pmatrix},
\]
where $*$ is a $1 \times (S-1)$ vector and $D_s$ is a $(S -1) \times (S - 1)$ diagonal matrix given by
\[
D_s = \mathrm{diag}\left(\frac{s^2 - s^1}{1+s^1}, \ldots, \frac{s^{S} - s^1}{1+s^1}\right).
\]
In the particular case $i = S+1$, (recall $e_{S+1}=0$), 
\[
DG_s(e_{S+1}) = \mathrm{diag}(s^1, \ldots, s^S).
\]
\section{Study for two species in a two-states environment}

\label{etudedim1}

We consider in this section $S=1$, i.e there is only two species in the community.  
The vectors field have the form:

\begin{equation}\label{onedimprosess}
\left\{
\begin{aligned}
&\frac{dX^1_t}{dt}=X^1_t\frac{s_t-s_tX^1_t}{1+s_tX_t^1}\\
&\frac{dX^2_t}{dt}=X^2_t\frac{-s_tX^1_t}{1+s_tX_t^1}\\
\end{aligned}
\right. 
\end{equation}

As before, since $X_t^1 + X_t^2 = 1$, we only study $(X^1_t)_{t>0}$, that we denote simply by $(X_t)_{t>0}$.  Thus we are interested in the study of 
\begin{equation}
\frac{d X_t}{dt} = s_{t} \frac{X_t(1-X_t)}{1+s_{t}X_t}.
\end{equation}

Moreover, we assume that $K=2$, the community evolves in two different environments. The selective parameters takes values in  $E=\{(s_1,0), (s_2,0)\}$.
It is possible to take more than two values for the fitness, and the following reasoning still hold. We restrict this study to the case of two values to simplify the notations.
Assume moreover $q_1:= q_{1,2}>0$ and $ q_2:=q_{2,1}>0$, and so $(s_t)_{t>0}$ has an unique invariant probability measure $\mu=p_1\delta_{s_1}+p_2 \delta_{s_2}$, where $p_1 = \frac{q_{2}}{q_{1}+q_{2}}$ and $p_2 = \frac{q_{1}}{q_{1}+q_{2}}$ . We set $M=[0,1] \times E$ and $M_0=\{0,1\}\times E$. It is easily checked that $\mathscr{P}_{erg}(M_0)=\{\mu_1,\mu_2\}=\{\delta_{0}\otimes \mu , \delta_{1} \otimes \mu \}$. To avoid trivial switching, we assume that $s_1 \neq s_2$, and without loss of generality, we assume that $s_1 > s_2$. For $i \in \{1,2\}$, we define $g_i : [0,1] \to \mathbb{R}$ by
\[
g_i(x) = s_{i} \frac{x(1-x)}{1+s_{i}x}.
\] We set 
\[
\Lambda_0 =  p_1 s_1 + p_2 s_2\quad \mbox{and} \quad \Lambda_1 = - \left( p_1 \frac{s_1}{1+s_1} + p_2 \frac{s_2}{1+s_2} \right). 
\]These quantities are the average growth rate of $(X_t)_{t>0}$ at $0$ and $1$, respectively. 
The following proposition gives the behavior of the process according to the signs of $\Lambda_0$ and $\Lambda_1$.

\begin{prop}
\label{prop:oneD}
We can describe four regimes :
\begin{enumerate}
\item If $\Lambda_0 < 0$, then $\Lambda_1 > 0$, and, for all $x \in (0,1)$ and $s \in E$, 
\[ 
\mathbb{P}_{x,s} \left( \limsup_{ t \to \infty} \frac{1}{t} \log ( X_t ) \leq \Lambda_0 \right) = 1.
\]
In particular, species $1$ goes extinct.
\item If $\Lambda_1 < 0$, then $\Lambda_0 > 0$, and, for all $x \in (0,1)$ and $s \in E$, 
\[ 
\mathbb{P}_{x,s} \left( \limsup_{ t \to \infty} \frac{1}{t} \log ( 1-X_t ) \leq \Lambda_1 \right) = 1.
\]
In particular, species $2$ goes extinct.
\item If $\Lambda_0 > 0$ and $\Lambda_1 > 0$,  there exists an unique invariant probability measure $\pi$ such that $\pi( \{0\} \times E) = \pi( \{1 \} \times E) = 0$. Moreover, $\pi$ is absolutely continuous with respect to the Lebesgue measure on $[0,1] \times E$ with explicitly computable density, and there exist $C, \theta, \gamma > 0$ such that, for all $(x,s) \in (0,1) \times E$ and all $t \geq 0$, 
\[
\| \mathbb{P}_{x,s} ( X_t \in \cdot) - \pi \|_{TV} \leq C ( (x)^{- \theta} + ( 1 - x )^{- \theta}) e^{ - \gamma t}.
\]
In particular, both species persist.
\item If $\Lambda_0=0$ or $ \Lambda_1 = 0$, then the only invariant probability measures of the process $\big((X_t,s_t)\big)_{t>0}$ are $\mu_1$ and $\mu_2$ . 

In particular, the process is not persistent.
\end{enumerate}

\end{prop}

\begin{proof}

We prove $(1)$, the proof of $(2)$ being the same as the one of $(1)$ by switching species $1$ and $2$.
Assume that $\Lambda_0 < 0$. In particular, $s_2 < 0$ and $p_1 s_1 < - p_2 s_2$, which implies that 
\[
- \Lambda_1 < - \left( \frac{1}{1+s_1} - \frac{1}{1+s_2} \right)p_2 s_2 < 0,
\]
proving that $\Lambda_1 > 0$. Now since $\Lambda_0 > 0$, \cite[Th. 3.1]{BS17}, implies that there exist $c > 0$  and $\eta > 0$ such that for all $x \in (0,c)$ and $s \in E$
\begin{equation}
\label{inqnobra1}
\mathbb{P}_{x,s}( \limsup_{t \to \infty} \frac{1}{t} \log(X_t) \leq \frac{\Lambda_0}{2}) \geq \eta.
\end{equation}
On the other hand, because $\Lambda_1 > 0$, there exists by Theorem 3.2 in \cite{BS17} $\varepsilon > 0$ such that for all $x \neq 0$,
\begin{equation}
\label{inqnobra2}
 \mathbb{P}_{x,s}( \tau < \infty) = 1,
\end{equation}
where $\tau = \inf \{ t \geq 0 \: : \:  X_t \leq 1 - \varepsilon \}$. Finally, because $0$ is a globally asymptotically stable equilibrium of $f^0$ on $[0,1)$, one can show that there exists a constant $C > 0$ such that for all $x \in [0, 1 - \varepsilon]$,
\begin{equation}
\label{inqnobra3}
  \mathbb{P}_{x,s} ( Z_t \in \mathcal{U} \times E) \geq C.
\end{equation}
Like in \cite[Th. 3.1]{BL16}, Equations \ref{inqnobra1}, \ref{inqnobra2} and \ref{inqnobra3} enable to conclude the proof of point 1.

We pass to the proof of point (3).  Since $\Lambda_0$ and $\Lambda_1$ are positive, \cite[Th. 3.2]{BS17} implies that for all $\varepsilon > 0$, there exists $1 > r > 0$ such that, for all $(x,s) \in (0,1) \times E$, almost surely 
\[
\liminf \frac{1}{t} \int_0^t \mathds1_{ r < X < 1 - r} ds \geq 1 - \varepsilon.
\]
This implies that the the sequence $(\Pi_t)_{t > 0}$ is almost surely tight in $(0,1)$. Moreover, every limit point of $(\Pi_t)$ is an invariant probability measure for $\big(X_t,s_t)\big)_{t>0}$ (see \cite[Th. 2.1]{benaimpersistence}). Thus, the process admits an invariant probability measure $\pi$ on $(0,1) \times E$. Uniqueness, absolute continuity and convergence in total variation easily follow from Theorem 4.10 in \cite{benaimpersistence} and Theorem 4.4 in \cite{BMZIHP}. 

Point (4) is proven in the following lemma as in \cite{HK19}.
\end{proof}

\begin{remarque}
Note that this property  still holds for any number of environments $K$, with $\Lambda_0=\sum\limits_{i\leq K}p_i s_i$ and $\Lambda_1=-\sum\limits_{i\leq K}p_i \frac{s_i}{1+s_i}$ .
\end{remarque}
In the following lemma, we describe more precisely the case where the growth rates are positive by computing explicitly the density of the invariant probability measure concentrated on $(0,1) \times E$.

\begin{lemma}
\label{lem:density}
Assume that $(U_t)_{t>0}$ admits an invariant probability measure $\Pi$ on $(0,1) \times E$. Then $\Pi$ is absolutely continuous with respect to the Lebesgue measure on $(0,1) \times E$. Moreover, denote by $h_1$ and $h_2$ the densities of $\Pi(\cdot,0)$ and $\Pi(\cdot,1)$, respectively. Then for all $x \in (0,1)$,
\[
h_i(x) = \frac{H(x)}{|g_i(x)|},
\]
where 
\[
H(x) = C (1-x)^{\beta \Lambda_1} x^{ \alpha \Lambda_0},
\]
with $ \alpha = \frac{q_1 + q_2}{|s_1 s_2|}$ and $\beta = \alpha (1+ s_1)(1+s_2)$ and $C$ the positive constant such that $\int_0^1(h_1 + h_2) = 1$. In particular, if $\Lambda_0=0$ or $\Lambda_1 = 0$, $U$ cannot admits such an invariant probability measure.
\end{lemma}

\begin{proof}
Let us assume that $(U_t)_{t>0}$ admits an invariant probability measure $\Pi$ on $(0,1) \times E$. This implies that $s_1 > 0$ and $s_2<0$. Indeed, by Proposition \ref{prop:oneD}, if $\Lambda_0 <0$, $X_t$ converges almost surely to $0$, and in particular cannot admit an invariant probability measure on $(0,1) \times E$. Thus we need $\Lambda_0 \geq 0$ and since we have assumed that $s_1 > s_2$, this implies $s_1 > 0$. For the same reason, $\Lambda_1$ must be non-negative, implying $s_2 < 0$. Thus, all the point in $[0,1]$ are accessible, yielding that the support of $\Pi$ has to be $[0,1] \times E$ (see Proposition 3.17 in \cite{BMZIHP}). Moreover, since for every $x \in (0,1)$, $f^1(x) > 0$, the weak bracket condition hold and $\Pi$ is unique and admits a density with respect to the Lebesgue measure (see \cite{BMZIHP} or \cite{BH12}). Moreover, this also implies by Theorem 1 in \cite{BHM15} that the densities $h_i$ are $C^{\infty}$ on $(0,1)$. Thus, $h_1$ and $h_2$ satisfy the Fokker-Planck equations (see e.g \cite[section 7.2]{BHM15}, \cite{BL16}, or \cite{DNW14}) :
\begin{equation}
    \begin{cases}q_1 h_1 - q_2 h_2 & = -( g_1 h_1)'\\
    q_1 h_1 - q_2 h_2 & = ( g_2 h_2)'
    \end{cases}.
    \label{eq:fokker}
\end{equation}
Now, one can check that the functions given above satisfy these equations. Moreover, since $h_1$ and $h_2$ are densities, they satisfy $\int_0^1 h_1 + h_2 = 1$; in particular they are integrable on $(0,1)$. This is the case if and only if $\Lambda_0 > 0$ and $\Lambda_1 > 0$. Therefore, if you assume that $\Lambda_0=0$ or $\Lambda_1=0$, $(U_t)_{t>0}$ cannot admit an invariant probability measure on $(0,1) \times E$, for otherwise it would have densities satisfying equations \ref{eq:fokker}, hence densities that are not integrable on $(0,1)$, a contradiction.
\end{proof}

\begin{ex}
\label{ex1}
Consider $p_1=p_2$, i.e the jump rate are the same for both environments. So, let simplify conditions 3) of property  \ref{prop:oneD} to obtain  the following conditions of persistence:

\begin{equation*}
  \left\{
      \begin{aligned}
        s_2&<\frac{-s_1}{1+2s_1}\\
        s_2&>-s_1
      \end{aligned}
    \right.
\end{equation*}

Note that if $s_1$ is smaller than $-0.5$ the first condition is automatically verified.

Then take the particular case  $s_1=1$. Then the previous condition to have persistence becomes $s_2<\frac{-1}{3}$.\\
To illustrate, we plot the comportment of the process for two values of $s_2$ close to $\frac{-1}{3}$, $-0.3$ and $-0.4$. 
\end{ex}

\begin{figure}[htbp]
\begin{minipage}[c]{.45\linewidth}
\begin{center}
\includegraphics[scale=0.35]{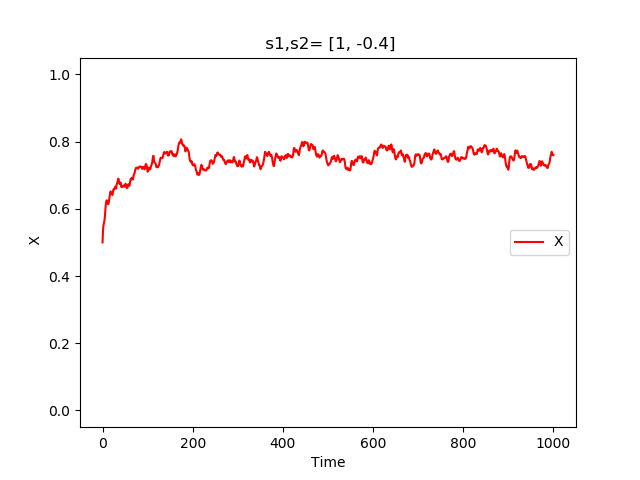}
\end{center}
\end{minipage}
\hfill
\begin{minipage}[c]{.45\linewidth}
\begin{center}
\includegraphics[scale=0.35]{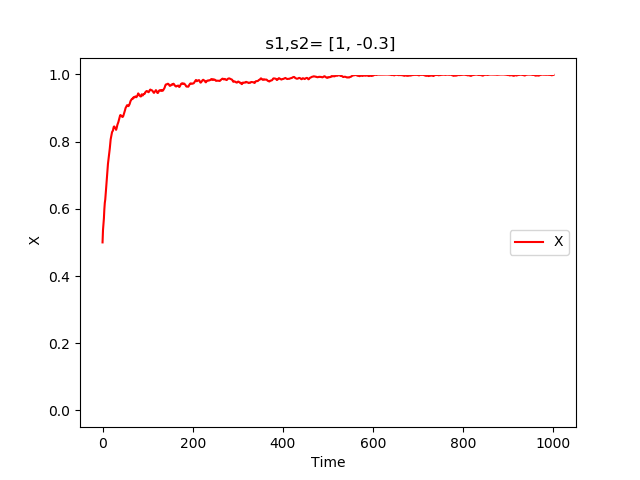}
\end{center}
\end{minipage}
\caption{Plots of the average of $(X_t)_{t>0}$ for different fitnesses, obtained by Monte Carlo method with 500 trajectories.}
\end{figure}

\begin{figure}[htbp]
\begin{minipage}[c]{.45\linewidth}
\begin{center}
\includegraphics[scale=0.35]{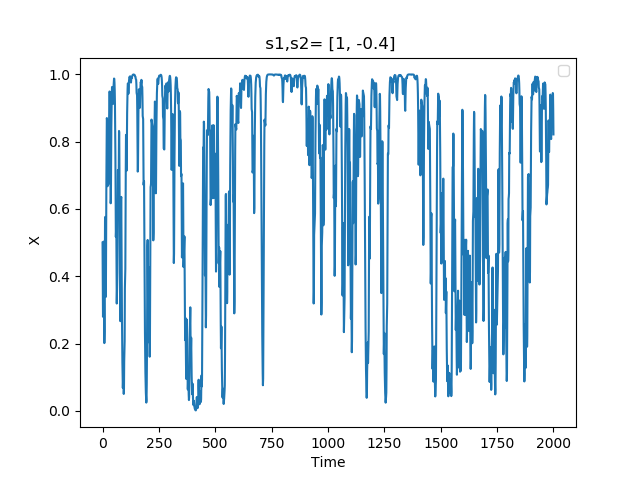}
\end{center}
\end{minipage}
\hfill
\begin{minipage}[c]{.45\linewidth}
\begin{center}
\includegraphics[scale=0.35]{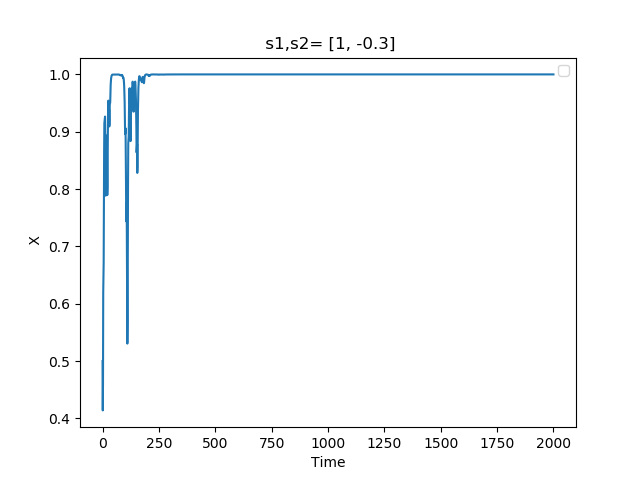}
\end{center}
\end{minipage}
\caption{Trajectory of  $(X_t)_{t>0}$  for two different fitness}
\end{figure}

Note that for $s_2=-0.4$ the process seems to be persistent whereas for $s=-0.3$ it seems to be absorbed quickly. So the numerical simulations are consistent with  the condition giving by Proposition  \ref{prop:oneD} .
\vspace{3mm}

\begin{ex}
This second example illustrate the case 4) of Proposition \ref{prop:oneD}.
Assume as in the previous example $p_1=p_2$ and take $s_1=-s_2$, so $\Lambda_0 = 0$ and $\Lambda_1 = \frac{2 s_1^2}{1 - s^2} > 0$ 

\begin{figure}[htbp]
\begin{minipage}[c]{.45\linewidth}
\begin{center}
\includegraphics[scale=0.35]{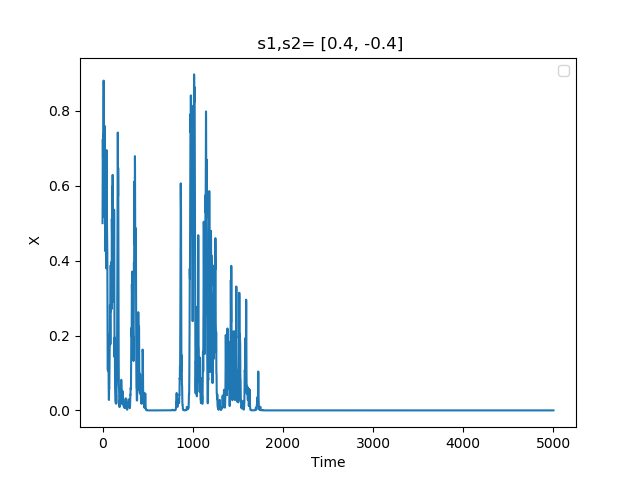}
\end{center}
\end{minipage}
\hfill
\begin{minipage}[c]{.45\linewidth}
\begin{center}
\includegraphics[scale=0.35]{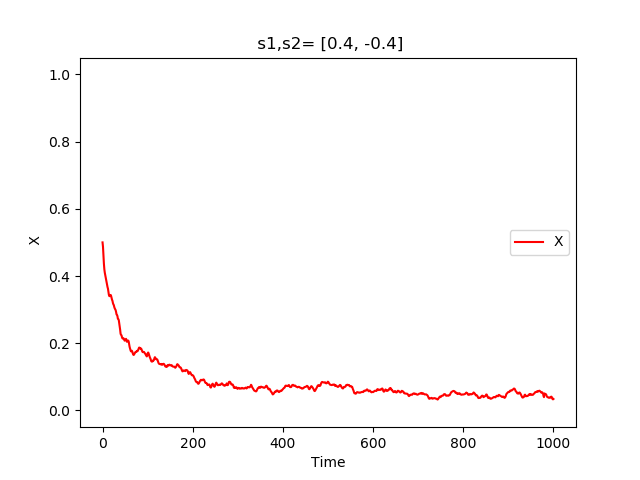}
\end{center}
\end{minipage}
\caption{Trajectories of $(X_t)_{t>0}$ and the average of $(X_t)_{t>0}$ for different fitness, obtain by Monte Carlo method with 500 trajectories. Parameters are $s_1=-s_2=0.4$, $q_{1,2}=q_{2,1}=1/2$. }
\end{figure}
As stipulated in Proposition  \ref{prop:oneD}, the process is not persistent. Moreover, the fact that $\Lambda_1$ is positive prevents species 2 from extinction. The numerical simulations suggest that the first species disappears.
 
\end{ex}

\begin{ex}
In this example, we illustrate Lemma \ref{lem:density}. Recall that if the parameters are such that the process is persistent, then the invariant distribution $\Pi$ on $M_+$ have explicit densities, given by 
\[
h_i(x) = \frac{C}{|s_i|} (1-x)^{\beta \Lambda_1 - 1}x^{\alpha \Lambda_0 -1}(1+s_i x),
\]
where $ \alpha = \frac{q_1 + q_2}{|s_1 s_2|}$ and $\beta = \alpha (1+ s_1)(1+s_2)$. In particular, it can exhibits several behaviour at the boundaries $0$ and $1$, according to the sign of $\beta \Lambda_1 - 1$ and  $\alpha \Lambda_0 -1$. Let us fix $s_2 = -0.2$, and jump rates $q_{1}=q_{2}=q$. In particular, $p_1 = p_2$ and the process is persistent if and only if $ 0.2 < s_1 < 1/3$. Fix $q = 1$. Then, it is easily seen that   $\beta \Lambda_1 - 1 > 0$ if and only if $s_1 < 1/4$; whereas $\alpha \Lambda_0 -1 > 0$ if and only if $s_1 > 1/4$. In particular, if $s_1 \in (0.2, 1/4)$, then $\lim\limits_{x \to 0}h_i(x) = +\infty$ and $\lim\limits_{x \to 1}h_i(x)=0$, whereas if $s_1 \in (1/4, 1/3)$, then we have the converse situation, i.e. $\lim\limits_{x \to 0}h_i(x) = 0$ and $\lim\limits_{x \to 1}h_i(x)= + \infty$. This is illustrated in Figure \ref{fig:densityvariousS}. Now fix $s_2 = -0.2$ and $s_1 =0.27$. Then, once again it is easy to check that  $\beta \Lambda_1 - 1 > 0$ if and only if $q > 0.259/5$; whereas $\alpha \Lambda_0 -1 > 0$ if and only if $q > 10/7$. Thus, we have three regimes :
 \begin{enumerate}
     \item if $q < 0.259/5$, then $h_i(x)$ goes to infinity both at 0 and 1;
     \item if $q \in (0.259/5, 10/7)$, then $h_i(x)$ vanishes at 0 and explodes at 1;
     \item if $q > 10/7$, then $h_i(x)$ vanishes both at 0 and 1.
 \end{enumerate}
 Plots of $h_1$ are presented in Figure \ref{fig:densityvariousq} for situation (1) and (3), situation (2) is plot on the right of Figure \ref{fig:densityvariousS}. This example shows that even if the process is persistent, the stationary distribution $\Pi$ certainly does not give mass to $0$ and $1$, but can be concentrated close to the extinction points. In the example with $s_1 = 0.27$, $s_2 = - 0.2$ and $q < 0.259/5$, the intuition is the following. Since $q$ is very small, the environment $s$ takes a really long time before changing. During this time, the process is getting really close to the boundary (say $0$ if we are following $s_2$), and spend a huge time here. When a switch occurs, the process goes away from $0$ fast enough, and come close to $1$ where it spends again a long time, and so on. In particular, it is much more likely that a switch occurs in the neighbourhood of $0$ or $1$, than in the middle. That is why, the process does not stay for a long time in the middle part, and concentrates near the boundaries. 
 
 \begin{figure}[htbp]
\begin{minipage}[c]{.45\linewidth}
\begin{center}
\includegraphics[scale=0.35]{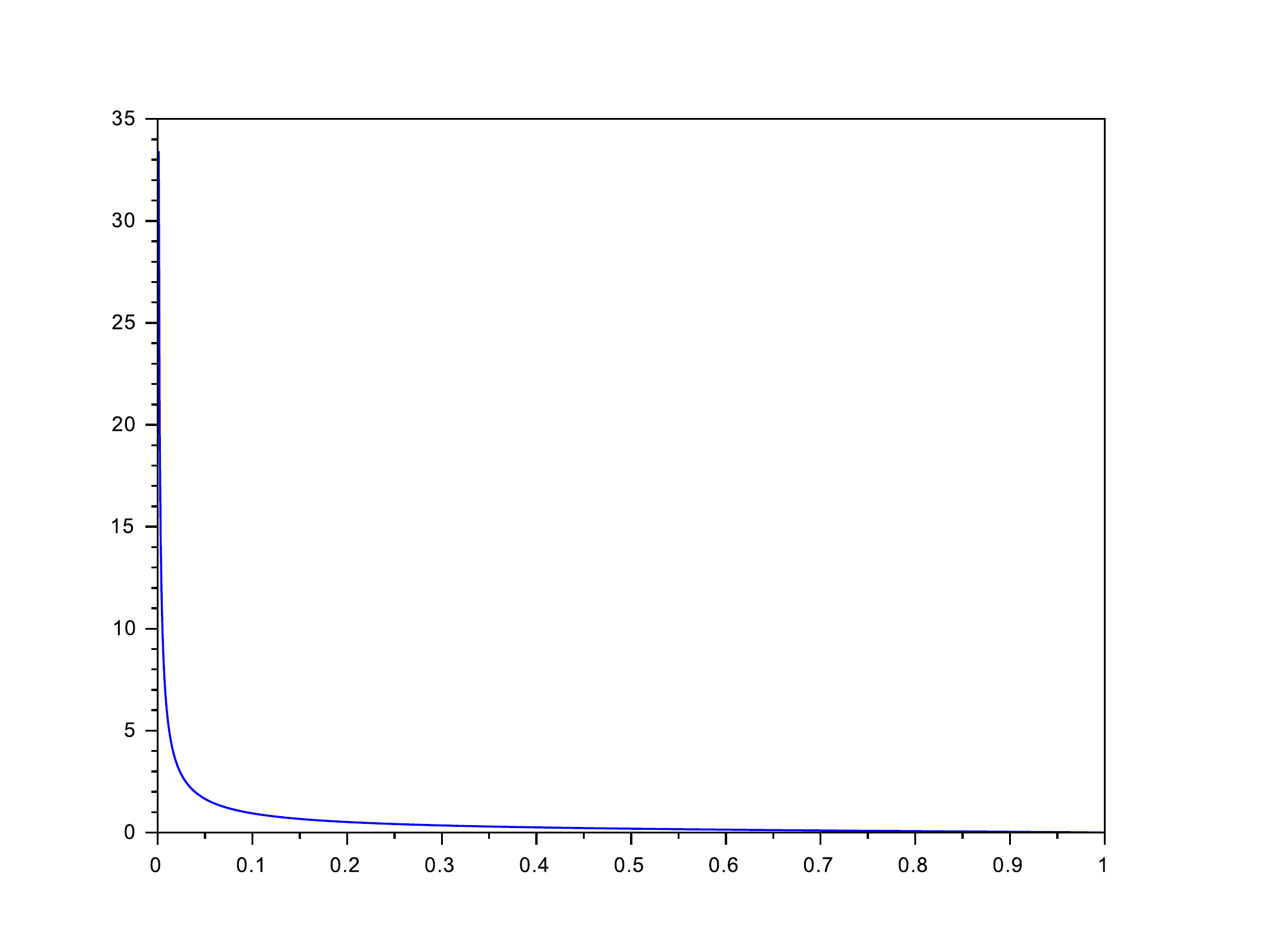}
\end{center}
\end{minipage}
\hfill
\begin{minipage}[c]{.45\linewidth}
\begin{center}
\includegraphics[scale=0.35]{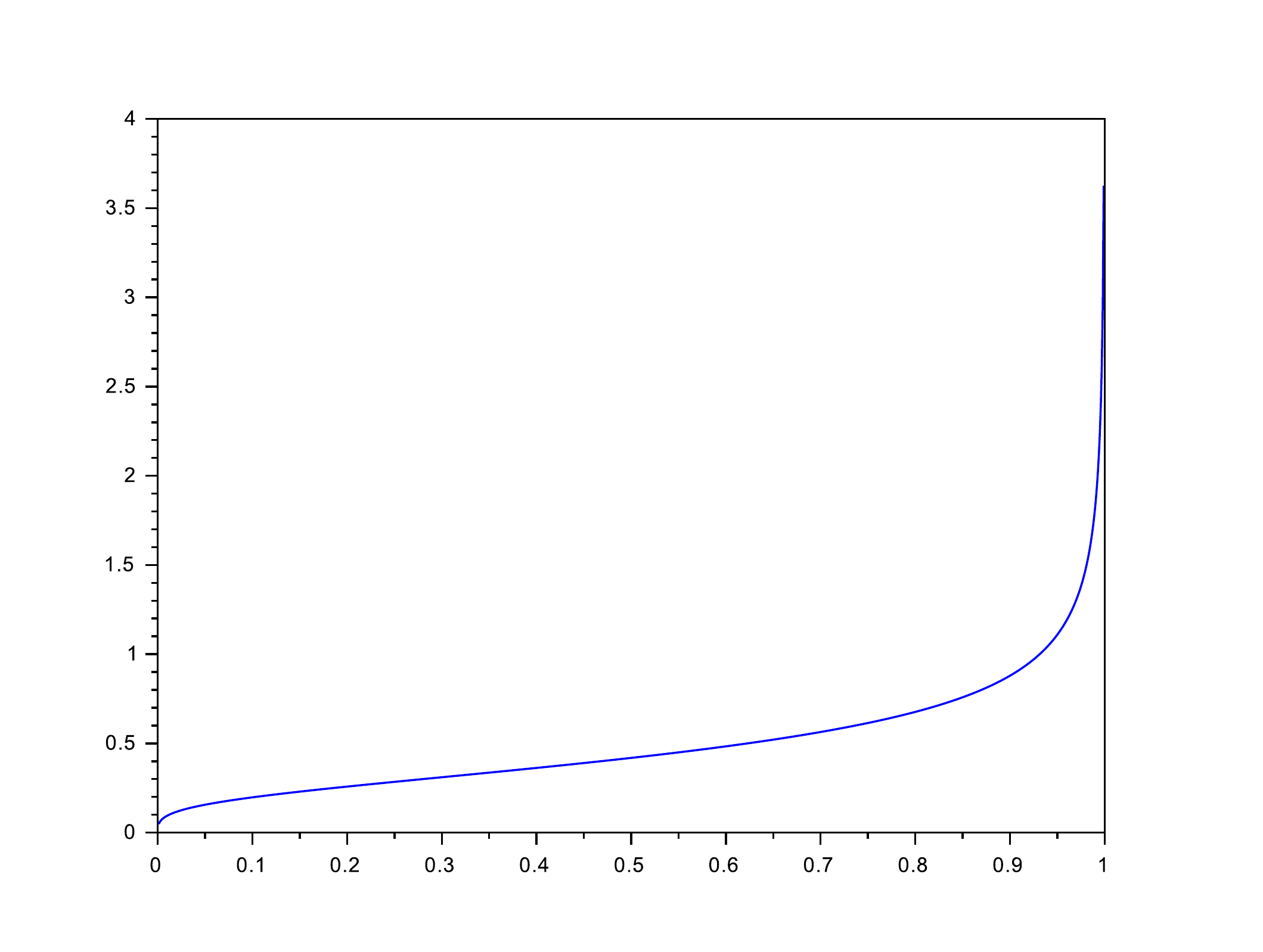}
\end{center}
\end{minipage}
\caption{Plot of $h_1$ for $q=1$ and $s_1 = 0.21$ (left) and $s_1 = 2.07$ (right)}
\label{fig:densityvariousS}
\end{figure}

 \begin{figure}[htbp]
\begin{minipage}[c]{.45\linewidth}
\begin{center}
\includegraphics[scale=0.35]{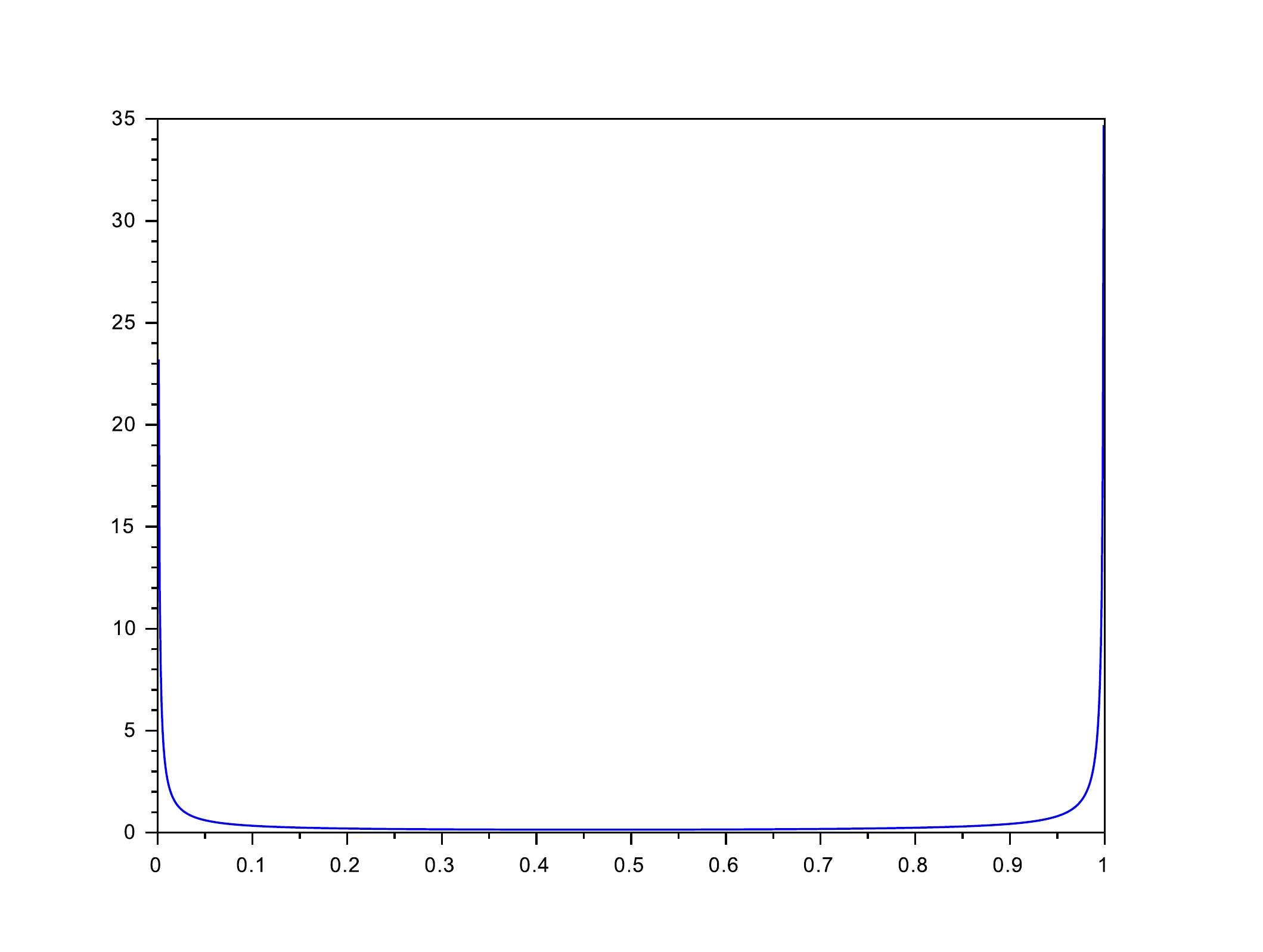}
\end{center}
\end{minipage}
\hfill
\begin{minipage}[c]{.45\linewidth}
\begin{center}
\includegraphics[scale=0.35]{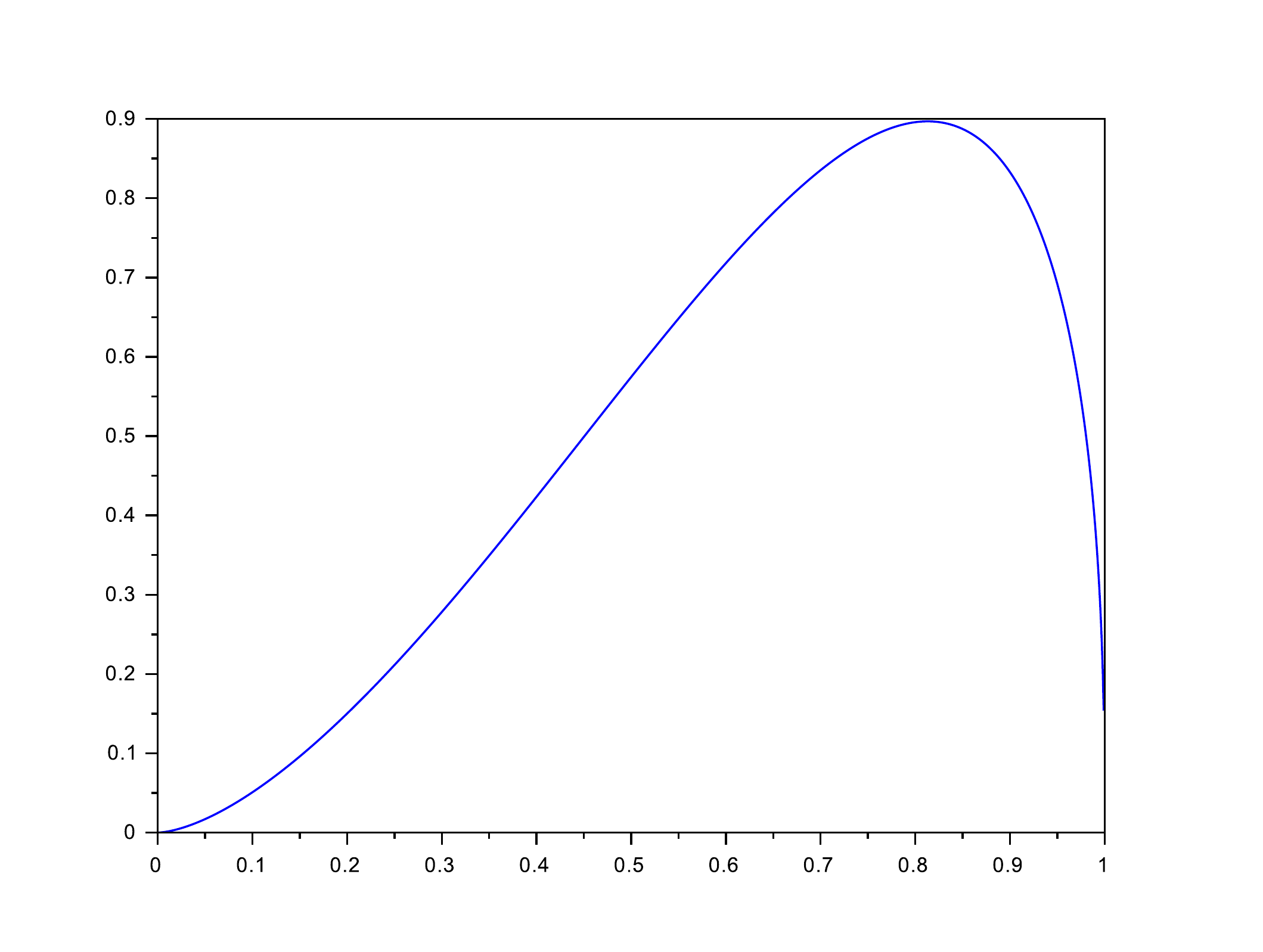}
\end{center}
\end{minipage}
\caption{Plot of $h_1$ for $s_1 = 0.27$  and $q=0.04$ (left) and $q=2$ (right)}
\label{fig:densityvariousq}
\end{figure}

\end{ex}
These examples concludes the study for two species. We now  generalise some  properties to a larger number of species and any number of environments. The space $M$  is no longer a line but a tetrahedron and the extinction set $M_0$ correspond to the face of $M$. Of course some intuitive behavior, as the fact that if a species has always the best fitness, it invades the community, is still true. But many arguments used in previous part to prove these results are specific to dimension one and does not hold anymore. 

\section{General results}
In this section we keep the notations of  part \ref{presentation}. We assume that the Markov chain $(s_t)_{t \geq 0}$ is irreducible on $E$. Hence, it admits a unique invariant probability measure on $E$, denoted by $\mu=p_1 \delta_{s_1} + \ldots + p_K \delta_{s_K}$.  

We describe now some behavior of the process in some remarkable environments.

\subsection{Sufficient conditions for a species to invade (or not) the community}
\label{sec:suffinvade}
In this section, we provide sufficient condition for a species to have a positive probability of invading the community. Furthermore, we prove that if one species have a positive probability to invade the community, then all the other ones cannot invade. For $i_0 \in \{1, \ldots, S+1\}$ , we set 
$$
\Lambda_{i_0} = \max_{  i \neq i_0 } \sum\limits_{j=1}^K  p_j \frac{s_j^i - s_j^{i_0}}{1+s_j^{i_0}} 
$$
 
\begin{theorem}
\label{thm:suffinvade}
Assume that for some $i_0 \in \{1, \ldots, S+1\},$ $\Lambda_{i_0}<0$. Then, for all $\alpha \in (\Lambda,0)$, there exist $\eta > 0$ and a neighbourhood $\mathscr{U}$ of $e_{i_0}$ such that, for all $x \in\mathscr{U}$ and all $\mathbf{s} \in E$,
\[
\mathbb{P}_{(x,\mathbf{s})} \left( \limsup\limits_{t \to \infty} \frac{1}{t} \log \|X_t - e_{i_0}\| \leq   \alpha \right) \geq \eta.
\]
Furthermore, for all $i \neq i_0$, we have $\Lambda_i>0$, and there exist $b > 1$, $\varepsilon > 0$, $\theta$ and $c>0$ such that, for all $x \in \mathscr{E} \setminus \mathscr{E}_0$, $s \in E$ and $i \neq i_0,$ 
\[
\mathbb{E}_{(x,s)}(e^{b \tau_i^{\varepsilon}}) \leq c ( 1 + \|x-e_i\|^{-\theta}),
\]
where
\[
\tau_{i}^{\varepsilon} = \inf\{ t \geq 0 \: : \|X_t - e_i \| \geq \varepsilon\}.
\]
\end{theorem}

\begin{proof}
Without loss of generality, we assume that $i_0 = 1$, and we set $\Lambda = \Lambda_1$. We set $M_0^1 = \mathscr{E}_0^1 \times E$ and $M_1^+ = M \setminus M_0^1$.

First, we note that $e_1 = (1, 0, \ldots, 0)$ is a common zero for all the $G_s$, $s \in E$, and that for all $i \geq 2$, $\mathscr{E}_0^i$ is a common invariant face for the process. Moreover, the Jacobian matrix of $G_s$ at $e_1$ is given by
\[
A_s = \begin{pmatrix}
\frac{-s^1}{1+s^1} & *\\
0 & D_s
\end{pmatrix},
\]
where $*$ is a $1 \times (S-1)$ vector and $D_s$ is a $(S-1) \times (S-1)$ diagonal matrix given by
\[
D_s = \mathrm{diag}\left(\frac{s^2 - s^1}{1+s^1}, \ldots, \frac{s^{S} - s^1}{1+s^1}\right).
\]
By assumption, $\Lambda  < 0$, which implies by Theorem 2.7 in \cite{S18} (or Proposition 2.5 in \cite{S18} and Theorem 3.5 in \cite{BS17}) that for all $\alpha \in (\Lambda,0)$, there exists $\eta > 0$ and a neighbourhood $\mathscr{U}$ of $e_{1}$ such that, for all $x \in \mathscr{U}$ and all $\mathbf{s} \in E$,
\[
\mathbb{P}_{(x,\mathbf{s})} \left( \limsup_{t \to \infty} \frac{1}{t} \log \|X_t - e_{1}\| \leq   \alpha \right) \geq \eta,
\]
which concludes the proof of the first assertion. 

We now prove the second assertion. Let $i \neq 1$. Since $\Lambda_1 < 0$, we have
\[
\sum\limits_{j=1}^K  p_j \frac{s_j^i - s_j^{1}}{1+s_j^{1}} < 0 
\]
As in Proposition \ref{prop:oneD}, we can show that this last inequality implies that
\[
\sum\limits_{j=1}^K  p_j \frac{s_j^1 - s_j^{i}}{1+s_j^{i}} > 0 
\]
and thus that $\Lambda_i > 0$.  Without loss of generality, we assume now that $i = S+1$. We set
\[
\Lambda_{S+1}^- = \min_{k \leq S} \sum\limits_{j=1}^K  p_j s_j^k,
\]
 and  we distinguish two cases. If $\Lambda_{S+1}^- > 0$, Proposition 2.5 in \cite{S18} and Theorem 3.5 in \cite{BS17} conclude the proof. If $\Lambda_{S+1}^- < 0$, we assume without loss of generality that there exists $k_0$ such that $\sum\limits_{j=1}^K   p_j s_j^k > 0$ for $k \leq k_0$ and $\sum\limits_{j=1}^K   p_j s_j^k < 0$ for $k > k_0$. We set
\[
\mathscr{E}_0^{1, \ldots, k_0} = \bigcap_{k = 1}^{k_0} \mathscr{E}_0^k,
\]
which is the set of extinction of the $k_0$-th first species. This set is invariant for all the vector fields $G_s$. Moreover, the Jacobian matrix of $G_s$ at $e_{S+1}=0$ is given by
\[
DG_s(e_i) =  \mathrm{diag}\left(s^1 , \ldots, s^{S}\right).
\]
Now, by definition of $k_0$, the assumptions of Theorem 2.8 in \cite{S18}, except for the accessibility of $0$ from $\mathscr{E}_0^{1, \ldots, k_0}$,  are satisfied. However, since \cite{S18}[Theorem 2.8 (iii)] is a local result, and since we get a  probability of convergence to $0$ which is bounded below in a neighbourhood of $0$ in $\mathscr{E}_0^{1, \ldots, k_0}$, a localisation argument similar to the one given in \cite{HS17} enables us to conclude. 
\end{proof}
This theorem states that when $\Lambda_{i_0} < 0$, species $i_0$ has a positive probability to invade the community, while if $\Lambda_{i_0} > 0$, species $i_0$ cannot invade the community. Furthermore, it is only possible to have one species satisfying $\Lambda_{i_0} < 0$. Let us finish by an interesting related result that up to now we can only pose as a conjecture.

\begin{conj}
If $\Lambda_{i_0}<0$, then for all $x$ with $x^{i_0} \neq 0$ and all $s \in E$, 
\[
\mathbb{P}_{(x,\mathbf{s})} \left( \limsup_{t \to \infty} \frac{1}{t} \log \|X_t - e_{i_0}\| \leq   \Lambda_{i_0} \right) = 1.
\]
\end{conj}

\subsection{The same species has always the best fitness}
\quad\\
\label{bestfitness}
In this section, we prove that if one of the species has always the best fitness, then this species take the upper hand on every over one. 
We have the following result, which generalizes Theorem \ref{th:deterministe}:

\begin{theorem}
\label{th:alwaysthebest}
Assume that for all $j \in \{1,\ldots,K\}$ and $i \in \{2, \ldots, S+1\}$, one has $s_j^1 > s_j^k$, and set
\[
\Lambda = \max_{2 \leq i \leq S+1 } \sum_{j=1}^K p_j \frac{s_j^i - s_j^1}{1+s_j^1} < 0.
\]
Then, for all $x \in \mathscr{E}$ with $x^1 > 0$ and all $s \in E$, one has
\[
\pp_{x,s}\left( \limsup_{t \to \infty} \frac{1}{t} \log \| X_t - e_1 \| \leq \Lambda \right) = 1.
\]
\end{theorem}

\begin{proof}
We use Theorem 3.8 in \cite{BS17}. By Theorem \ref{thm:suffinvade}, we know that  for all $\alpha \in (\Lambda,0)$, there exists $\eta > 0$ and a neighbourhood $\mathscr{U}$ of $e_{i_0}$ such that, for all $x \in \mathscr{U}$ and all $\mathbf{s} \in E$,
\[
\mathbb{P}_{(x,\mathbf{s})} \left( \limsup_{t \to \infty} \frac{1}{t} \log \|X_t - e_{i_0}\| \leq   \alpha \right) \geq \eta.
\]
Furthermore, since $e_1$ is an asymptotically stable equilibrium whose basin of attraction is $\mathscr{E}_1^+$ for each $G_s$, we deduce that the point $e_1$ is accessible from $\mathscr{E}_1^+$ for the PDMP. This means that each neighbourhood of $e_1$ can be reached with positive probability by the process (see \cite{BMZIHP} for a precise definition). 
Since the set $\mathscr{E}_1^+$ is not compact, we have to study the behavior of the process near $\mathscr{E}_0^1$ to conclude that the process converges to $e_1$ with probability one from everywhere in $\mathscr{E}_1^+$. For all $s \in E$, we have
\[
F^1_s(x) = \frac{s^1 - \sum\limits_{i=1}^S s^ i x^i}{1 + \sum\limits_{i=1}^S s^ i x^i}
\]
By assumption, for all $x \in \mathscr{E}_0^1$, one has $F^1_s(x) > 0$. This implies that for all $\mu \in \mathscr{P}_{erg}(\mathscr{E}_0^1)$, $\lambda_1(\mu) > 0$. In particular, the process is $H$ - persistent with respect to $\mathscr{E}_0^1$. Therefore, by \cite[Proposition 8.2]{benaimpersistence}, there exists a Lyapunov function for the process near $\mathscr{E}_0^1$ and thus Hypothesis 3.7 in \cite{BS17} is satisfied for the set $M_1^+$. This concludes the proof by Theorem 3.8 in \cite{BS17}.
\end{proof}

\subsection{A species is always disadvantaged with respect to the other species in each environment}

\quad\\
With no loss of generality, we assume the species $1$ is always disadvantaged with respect to the other species. More precisely for each environment, the species 1 has a negative fitness. We assume moreover that the last species is not extinct, i.e $1-\sum\limits_{k=1}^Sx^k>0$.

We now prove in this situation that the species 1 is going to 0. The strategy is  to show that the vector field is always entering in  $E_{b}=\{(x^1,...,x^S)\in E\: :  x_1<b(1-\sum\limits_{k=2}^Sx^k)\}\}$, for all $b \in ]0,1[$. Hence, if the process enters in $E_b$, it cannot escape $E_b$. Then, we prove, for all $ b \in ]0,1[$, the hitting time of $E_b$ is finite.

Let us consider the hyperplane of the form $\Delta_b :\{(x^1,...,x^S) \in E \: :x_1=b(1-\sum\limits_{k=2}^Sx^k)\}$ for $b$ in $]0,1[$. This is the hyperplane passing through the points $be_1,e_2,...,E$, where $e_i$ is the $i^{th}$ vector of the natural base, and $E_b$ is the area under $\Delta_b$. The following proposition proves that  for each environment, the vector field is entering in $E_b$.

\begin{prop}
Let $\mathbf{s}_i$  be such that  $s^1_i<0$ for $i \in {1,..,K}$,  then for all $b$ in $]0,1[$,  $E_b$ is a trap area for the process $X_t$. 
\end{prop}

\begin{proof}
Consider a fixed environment and note $s^k$ the fitness of the $k^{th}$ species.  
Let us start by remarking that $v_b=(1,b,...,b)$ is an orthogonal vector of the hyperplane $\Delta_b$, pointing outward $E_b$. 
Now, we look at the sign of  projection of the vector field $\big(dX^1,..., dX^S\big)$ at the point $\delta_b=\big(b(1-\sum\limits_{k=2}^Sx^k),x^2,...,x^S\big)$ of $\Delta_b$ on $v_b$. If the result is strictly negative for all $x$ and $b$ in $]0,1[$, along  $\Delta_b$, the vector field is entering in $E_b$.\\
\begin{align*}
    &\Big\langle\big(dX^1,...,dX^S\big)\big(\delta_b\big);v_b \Big\rangle<0\\
    &\Leftrightarrow b\left(1-\sum\limits_{k=2}^Sx^k\right)\left(s^1-\sum\limits_{k=1}^Ss^k x^k\right)+b\left(\sum\limits_{k=2}^Sx^k\Big(s^k-\sum\limits_{k=1}^S x^ks^k\Big)\right)<0\\
    &\Leftrightarrow bs^1\left(1-\sum\limits_{k=1}^Sx^k\right)<0\\
    &\Leftrightarrow s^1<0
\end{align*}

Therefore, the vector field is entering in $E_b$ if and only if $s^1<0$. But we assumed $s^1_i<0$ for each $i$, so  for all environments the vector field is entering in $E_b$.  This concludes the proof.
\end{proof}

\begin{cor} 
\label{droiteextinctioncorr}
Assume that $s_i^1<0$ for each $i$. Then the process $(X_t)_{t>0}$, starting at $(x^1, \ldots, x^{S+1})$ such that $1 - \sum\limits_{k=1}^S x^k > 0$, verifies $\lim\limits_{t\rightarrow\infty} X^1_t=0$.
 In particular  species $1$ goes to extinction.
\end{cor}
\begin{proof}

Denote now by $b_t$ the intersection between the hyperplane passing through the points $X_t,e_2,...,E$ and the straight line directed by $e_1$. By the previous propriety ,  for all $t$ , $E_{b_t}$ is a trap area.
Moreover the species $1$ has in each environment a strictly negative fitness, so by the previous calculation the vector field on $\Delta_{b_t}$ is strictly entering. Consequently, $b_t$ is strictly decreasing and $b_t$ converges to $b_{lim}$. Note that by definition and the previous proposition, for all $t \geq 0$, $X_t^1 \leq b_t$. Hence, if $b_{lim}=0$,  $X^1_t$ converges to $0$.

If   $b_{lim}\neq 0$, the process converges to $\Delta_{b_{lim}}$. Necessarily the process converges to an invariant area included in $\Delta_{b_{lim}}$. 
But according to the previous proposition, on each point of $\Delta_{b_{lim}}$ where $\left(1-\sum\limits_{k=1}^Sx^k\right)\neq 0$, all the vector fields are strictly entering in $E_{b_{lim}}$. So  this invariant area is also included in $\left(1=\sum\limits_{k=1}^Sx^k\right)$, which implies $X^1 = b_{lim} X^1$. Since $b_{lim} \neq 0$, we conclude that $X^1 = 0$. Thus, species one goes to extinction in all cases.

\end{proof}
\begin{remarque}
A immediate corollary of this property  is, if a species has always the best fitness in each environment, then it invades the community. Note that Theorem \ref{th:alwaysthebest} above proves this fact independently and gives the rate of convergence.
\end{remarque}
 
Let us now give an example to illustrate this property .
\begin{ex}
Consider the case of three species and two environments, we take the notation of part \ref{3species}.\\
Assume the species $Y$ has always a negative fitness, and the fitness are ordered like :
\begin{center}
$
 \left\{
    \begin{array}{ll}\label{fitdefa}
       s_1>0>r_1 \\
       0>r_2>s_2
    \end{array}
\right.
$

\end{center}
In Figure \ref{fig:exYnegfit} are plotted the phase portrait of the vectors fields for each environment.

\begin{figure}[htbp]
\label{fig:exYnegfit}
\begin{minipage}[c]{.45\linewidth}
\begin{center}
\includegraphics[scale=0.35]{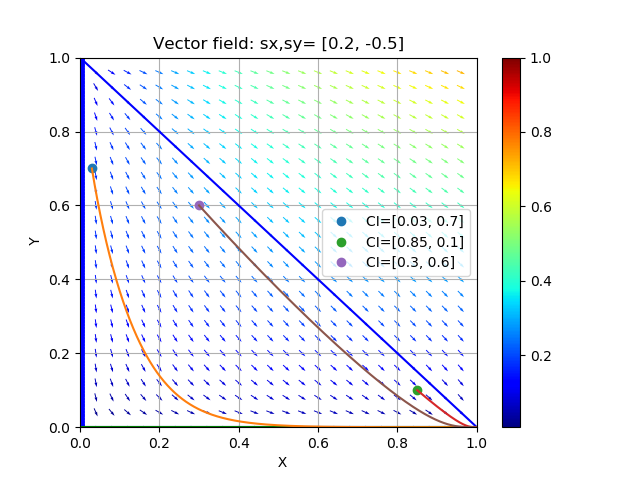}
\end{center}
\end{minipage}
\hfill
\begin{minipage}[c]{.45\linewidth}
\begin{center}
\includegraphics[scale=0.35]{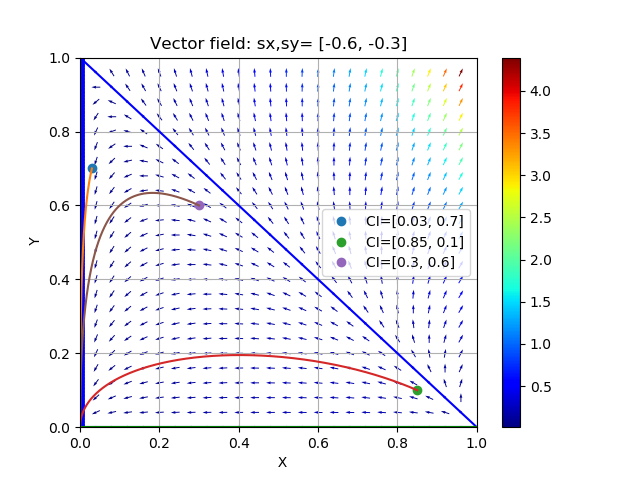}
\end{center}
\end{minipage}
\caption{Vectors fields and trajectories for the previous configuration of fitness  }
\end{figure}
In this situation $\Delta_b$ is a straight line in red on figure \ref{champdroite2} and the $v_b$ correspond to the direction of the green straight line.

\begin{figure}[htbp]
    \centering
    \includegraphics[scale=0.35]{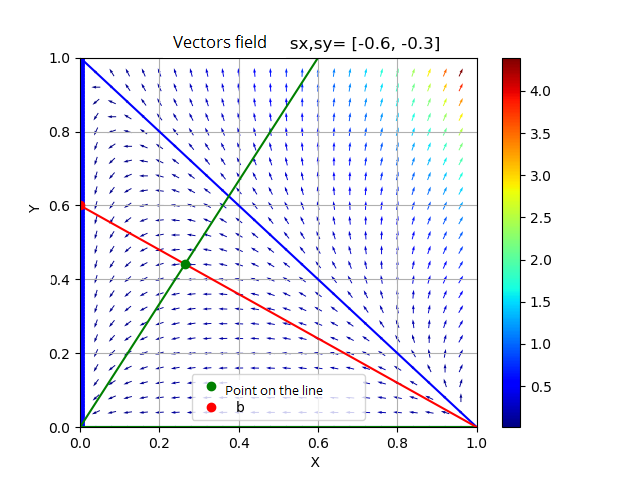}
    \caption{Illustration  of $\Delta_b$ and $v_b$.}
    \label{champdroite2}
\end{figure}
Figure \ref{champdroite2} shows that the field is entering in $E_b$, and this true for each $b$. The conclusion follows, the $Y$ species goes to extinction. The following figure illustrate a trajectory where the species $Y$ has always a negative fitness.

\begin{figure}[htbp]
\begin{center}
\includegraphics[scale=0.35]{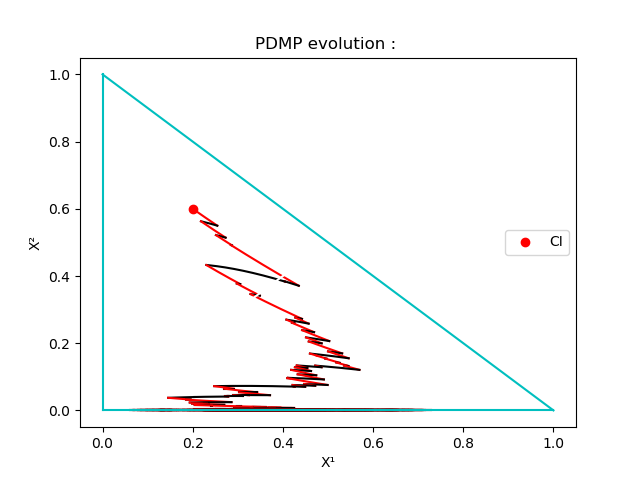}
\end{center}
\caption{A PDMP trajectory where the $2^{nd}$ species goes to extinction. The values of environments are $s_1=[0.4,-0.1]$ and $ s_2=[-0.3,-0.2]$.}
\end{figure}

\end{ex}

\vspace{50mm}

We  now look at the case of three species and two  jumps rates.
Previous theorems  find their application in this particular case. The two dimensional study made in Section \ref{etudedim1} is also necessary to understand  the long time  behavior of the process. The space of extinction $ M_0 $ corresponds here to the side of the triangle defined by the apex $ (0,0), (1,0), (0,1) $. Thus understanding the long time behavior  requires knowledge of the invariant measures of the process on $ M_0 $. In particular the behaviors of the process on each side of the triangle, which corresponds to studying the process with only two species.\\

\section{The case of 3 species }

\subsection{Introduction}
\label{3species}
 For sake of simplicity, when there are only three species, we use the notations $X$, $Y$ and $Z$, instead of $X^1$, $X^2$, $X^3$. We also denote by $s_k$ and $r_k$ the fitness in environment $k\in \{1,...,K\}$ of species $X$ and $Y$, respectively (remember that we have set the fitness of the third species to be $0$). Since we have for all $t \geq 0$, $X_t + Y_t + Z_t= 1$, we are interested in the following equations :
 
\begin{equation}\label{twodimprosess}
\left\{
\begin{aligned}
&\frac{dX_t}{dt}=X_t\frac{s_t-s_t X_t-r_t Y_t}{1+s_t X_t+r_t Y_t}\\
&\frac{dY_t}{dt}=Y_t\frac{r_t-s_t X_t-r_t Y_t}{1+s_t X_t+r_t Y_t}\\
\end{aligned}
\right. 
\end{equation}

The community still evolves in two different environments. The fitness takes values in  $E=\{\mathbf{s}_1=(s_1,r_1,0),..., \mathbf{s}_K=(s_K,r_K,0)\}$.
 To avoid trivial switching, we assume that $\mathbf{s}_i \neq \mathbf{s}_j$ if $i\neq j$. Assume yet that  $(\mathbf{s}_t)_{t>0}$ has a unique invariant probability measure $\mu=\sum\limits_{i=1}^{K}p_i\delta_{\mathbf{s}_i}$. 
The set of ergodic measures of the process on $M_0$, $\mathscr{P}_{erg}(M_0)$,  depends on the values of the environment.
This set always contains $\mu_3 = \delta_{(0,0)} \otimes \mu$,  $ \mu_2=\delta_{(0,1)} \otimes \mu$ and $ \mu_1=\delta_{(1,0)} \otimes \mu$, the Dirac masses on each vertex. But on each side, it may have one other ergodic measure. It corresponds to persistence of the two remaining species on this  side.
Let's give more details about it:

\begin{prop}
\label{prop:ergodic}
Let define 
\begin{align*}
    &\Lambda_{0}^1=\sum\limits_{i=1}^K p_i r_i,\quad \Lambda_{1}^1=\sum\limits_{i=1}^K -p_i\frac{r_i}{1+r_i} \\
    &\Lambda_{0}^2=\sum\limits_{i=1}^K p_is_i \quad \Lambda_{1}^2=\sum\limits_{i=1}^K -p_i\frac{s_i}{1+s^i} \\
    & \Lambda_{0}^3=\sum\limits_{i=1}^K p_i\frac{r_i-s_i}{1+s_i}, \quad \Lambda^3_{1}=\sum\limits_{i=1}^K p_i\frac{s_i-r_i}{1+r_i}\\
\end{align*}

 Then the process admits an  ergodic measure, $\nu_i$ on $\mathrm{int}(\mathscr{E}^i_0) \times E$ if and only if $\Lambda_0^i>0$ and $\Lambda_1^i>0$. Furthermore, this ergodic measure is unique and explicitly computable.
\end{prop}
\begin{proof}
This property  is a corollary of Proposition \ref{prop:oneD}.
\end{proof}
When $\Lambda_0^i$ and $\Lambda_1^i$ are positive, we denote by $\nu_i$ the unique ergodic measure on $\mathrm{int}(\mathscr{E}^i_0) \times E$.

\begin{remarque}
The signs of $\Lambda_0^i$ and $\Lambda_1^i$ determines the behavior of the process on side $i$ in a neighbourhood of the extinction and the invasion, respectively, for the $i^{th}$ species.
\end{remarque}

\subsection{Study for two environmental states}
\quad\\
Now we assume $K=2$, then $\mu=p_1\delta_{\mathbf{s}_1}+p_2\delta_{\mathbf{s}_2}$ with  $p_1 = \frac{q_{2}}{q_{1}+q_{2}}$ and $p_2 = \frac{q_{1}}{q_{1}+q_{2}}$

We now describe the different possible regimes according to the values of the parameters.
\begin{itemize}
    \item If a species,  for example $Y$, is always disadvantaged with respect to the same other species in each environment, by Corollary \ref{droiteextinctioncorr}, this species goes to extinction. 
    
Reorganising  the order of species if necessary, it corresponds to   \\
$
 \left\{
    \begin{array}{ll}\label{fitdefa}
       s_1>0>r_1 \\
       0>r_2>s_2
    \end{array}
\right.
$
or $
 \left\{
    \begin{array}{ll}
       s_1>0>r_1 \\
       0>s_2>r_2
    \end{array}
\right.
$
or
$
 \left\{
    \begin{array}{ll}
       s_1>0>r_1 \\
       s_2>0>r_2
    \end{array}
\right.
$
or
$
 \left\{
    \begin{array}{ll}
       s_1>0>r_1 \\
       s_2>r_2>0
    \end{array}
\right.
$

Then, there are several possibilities for the two remaining species, depending on the behavior of the process on the axis $\{y=0\}$. If we are in the situation 1 or 2  of the Proposition \ref{prop:oneD}, then a species will invade the community. However, if we are in the situation 3 of Proposition \ref{prop:oneD}, there is persistence of both other species (see Figure \ref{fig:YextXZpers}).

\begin{figure}[htbp]
\label{fig:YextXZpers}
\begin{center}
\includegraphics[scale=0.35]{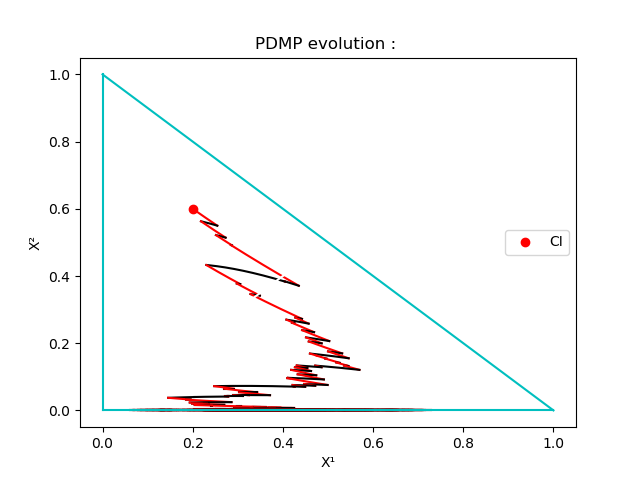}
\end{center}
\caption{A PDMP trajectory where species $Y$ goes to extinction and  both other species persist. The values of environments are $s_1=[0.4,-0.1]$ and $ s_2=[-0.3,-0.2]$.}
\end{figure}

\item Else, reorganising if necessary the order of species, we are in the following situation :\\

\begin{equation}
     \left\{
    \begin{array}{ll}
       s_1>0>r_1 \\
       r_2>0>s_2
    \end{array}
\right.
\label{ordrefitness}
\end{equation}

We see in the following that several behaviors are possible. According to the values of the fitness, each species, even the neutral, can invade the community (see Theorem \ref{thm:neutralinvade}). However, it is not possible to have persistence of all the species.
\end{itemize}

Before describing the possibilities for the case \eqref{ordrefitness}, we compute the invasion rates for each species.

\subsubsection{Computation of the invasion rates}
Referring to Theorem \ref{thm:persistence}, it is important to describe the set of the ergodic measures on $M_0$ and to compute, for each species, the associated invasion rate. We know that there are at least three ergodic probability measures on $M_0$, namely $\mu_1$, $\mu_2$ and $\mu_3$. According to Proposition \ref{prop:ergodic}, there are at most three other ergodic probability measures on $M_0$, denoted by $\nu_1$, $\nu_2$ and $\nu_3$. 
We now compute the invasion rates of species $i$ with respect to each of these ergodic measures. We detail the computations for species 1, they are similar for the other species. 

\begin{lemma}
\label{lem:invasionrates}
We have
\[ \lambda_1(\mu_3) = \sum\limits_{i=1}^K p_i s_i=\Lambda_0^2, \quad \lambda_1(\mu_2) = \sum_{i=1}^K p_i \frac{s_i - r_i}{1+r_i}=\Lambda_1^3, \quad \lambda_1(\mu_1) = 0.
\]
Moreover, when $\nu_1, \nu_2$ and $\nu_3$ exist, we have
\[
\lambda_1(\nu_3) = \lambda_1(\nu_2) = 0\]
and
\[
\lambda_1(\nu_1) = \left( \frac{s_0}{|r_0|} + \frac{s_1}{|r_1|}\right)C_1,
\]
where $C_1$ is an explicitely computable positive constant.
\end{lemma}

\begin{proof}

Recall that for all $\mu \in Perg(M_0)$, $\lambda_1(\mu) = \sum\limits_{i=1}^K \int_{M_0} F^i(x,y) d \mu^i(x,y)$, with 
\[
F^i(x,y) = \frac{s_i - (s_i x + r_i y)}{1 + s_i x + r_i y}.
\]
In particular, the formulae for $\lambda_1(\mu_3)$ and $\lambda_1(\mu_2)$ are immediate from the definitions of $\mu_2$ and $\mu_3$.  The fact that $\lambda_1(\mu_1) = \lambda_2(\nu_2) = \lambda_3(\nu_3) = 0$ is straightforward from \cite[Theorem 5.1 (i)]{benaimpersistence}. It remains to compute $\lambda_1(\nu_1)$. First, we note that by Proposition \ref{prop:ergodic}, $\nu_1$ only exists if $p_0 r_0 + p_1 r_1 > 0$ and $ p_0 \frac{r_0}{1+r_0} + p_1 \frac{r_1}{1+r_1} < 0$. In particular, it can only exists if $r_0$ and $r_1$ are of opposite signs. Next, we know by Lemma \ref{lem:density} that in this case, $\nu_1$ admits a density with respect to the Lebesgue measure on $M_0^1$, which is given by 
\[
g_i(y) = \frac{G(y)(1+r_i y)}{y(1-y)|r_i|}. 
\]
for some function $G$ similar to the function $H$ given in Lemma \ref{lem:density}.
Thus, 
\begin{align*}
\lambda_1(\nu_1) & = \int_0^1 F^0(0,y)g_0(y) + F^1(0,y)g_1(y) dy\\
& = \int_0^1 \left( \frac{s_0 - r_0 y}{|r_0|} + \frac{s_1 - r_1 y}{|r_1|} \right) \frac{G(y)}{y(1-y)}dy
\end{align*}
Now, since $r_0$ and $r_1$ have opposite signs, we have for all $y \in (0,1)$ 
\[
\frac{s_0 - r_0 y}{|r_0|} + \frac{s_1 - r_1 y}{|r_1|} = \frac{s_0}{|r_0|} + \frac{s_1}{|r_1|},
\]
which implies that
\[
\lambda_1(\nu_1) = \left( \frac{s_0}{|r_0|} + \frac{s_1}{|r_1|}\right) \int_0^1 \frac{G(y)}{y(1-y)} dy. 
\]
and concludes the proof.
\end{proof}
The invasion rates for species 1, as well as for species 2 and 3 are summed up in Table \ref{tab:invasion}. As for $C_1$, $C_2$ and $C_3$ are positive constant that may be computed. However, since it is sufficient to know the signs of the invasion rates, it does not really matter to have the exact expression of the $C_i$.

\begin{table}
\begin{center}
\begin{tabular}{c|c|c|c||c|c|c}
       & $\mu_1$ & $\mu_2$ & $\mu_3$ & $\nu_1$ & $\nu_2$ & $\nu_3$ \\ \hline
      1 & 0 & $\Lambda_1^3 $ & $\Lambda_0^2$ & $C_1 ( \frac{s_1}{|r_1|} + \frac{s_2}{|r_2|})$ & 0 & 0\\ \hline
      2 & $\Lambda_0^3$ & 0 & $\Lambda_0^1$ & 0 & $C_2 ( \frac{r_1}{|s_1|} + \frac{r_2}{|s_2|})$ & 0\\ \hline
      3 & $\Lambda_1^2$ & $\Lambda_1^1$ & 0 & 0 & 0 &  $C^3 ( -\frac{r_1}{|s_1-r_1|} - \frac{r_2}{|s_2-r_2|})$
    \end{tabular}
    
    \end{center}
    \caption{Invasion rates}
    \label{tab:invasion}
  \end{table}  
\subsection{Environment favours and disadvantages alternately two species, the other one stay neutral over time.}

With no loss of generality, we assume the  environment promotes and disadvantages alternately the first and the second species. So without regard to order, only one configuration is possible:
\begin{equation}
         s_1>0>r_1, \qquad
       r_2>0>s_2.
\end{equation}

In this configuration, many different situations are possible for the long time behavior. In the following we highlight some of them. Mainly, we discuss about the situation where the neutral species (the third in the last configuration) invade the community. Then we prove that persistence of the three species is not possible with only two environments.

\subsubsection{The neutral species invades the community}
In this section, we use results of Section \ref{sec:suffinvade} to provide sufficient conditions for the invasion of the neutral species. We also give an example of parameters satisfying theses conditions.

\begin{theorem}
\label{thm:neutralinvade}
Assume that $\Lambda_0^1<0$ and $\Lambda_0^2<0$. Set 
\[ \Lambda = \max(\Lambda_0^1, \Lambda_0^2) < 0
\]
Then, for all $(x,y) \in \mathscr{E} \setminus \mathscr{E}_0^3$, and all $\mathbf{s} \in E$,
\[
\mathbb{P}_{(x,y,\mathbf{s})} \left( \limsup_{t \to \infty} \frac{1}{t} \log \|(X_t,Y_t)\| \leq \Lambda \right) =1.
\]
\end{theorem}
\begin{proof}
With the notations of Theorem \ref{thm:suffinvade},  $\Lambda = \Lambda_3$, and thus, by this theorem,
for all $\alpha \in (\Lambda, 0)$, there exists $\eta > 0$ and a neighbourhood $\mathscr{U}$ of $0$ such that, for all $(x,y) \in \mathscr{U}$ and all $\mathbf{s} \in E$,
\[
\mathbb{P}_{(x,y,\mathbf{s})} \left( \limsup_{t \to \infty} \frac{1}{t} \log \|(X_t,Y_t)\| \leq \alpha \right) \geq \eta.
\]
Now, we prove that the point $0$ is accessible from $\mathscr{E}_+^3$. By Lemma \ref{lem:accessconjug}, $0$ is accessible if and only if it is accessible for the vector fields given by $g_s(x,y) =( x(s^1 - (s^1 x + r^1y)), y(r^1 - (x^1x + r^1y)))$. Consider the convex combination $g_p = p_1 g_{s_1} + p_2 g_{s_2}$. Then $g_p = g_{\mathbf{s}_p}$, where $\mathbf{s}_p = p_1 \mathbf{s}_1 + p_2 \mathbf{s}_2$. Now, since $\Lambda_0^1$ and $ \Lambda_0^2$ are negative,  $\mathbf{s}_p = (s_p, r_p, 0)$. Hence, by Theorem \ref{th:deterministe} and Lemma \ref{lem:accessconjug}, the flow generated by $g_p$ converges to $0$ for all initial condition in $\mathscr{E}^3_+$. In particular, $0$ is accessible from $\mathscr{E}^3_+$ with the vector fields $g_s$, hence it is accessible for $\big((X_t, Y_t)\big)_{t>0}$. 

Finally, as for the proof of Theorem \ref{th:alwaysthebest}, we have to show that the face $\mathscr{E}_0^3$ is repulsive. We use Theorem \ref{thm:persistence}. We know that $\mathscr{P}_{erg}(\mathscr{E}_0^3)$ contain $\mu_1$, $\mu_2$ and possibly $\nu_3$. Furthermore, by Lemma \ref{lem:invasionrates}, Proposition \ref{prop:oneD} and the assumption that $\Lambda_0^1<0$ and $\Lambda_0^2<0$, we have that
\[
\lambda_3(\mu_1) = \Lambda_1^2 > 0 \quad \mbox{and} \quad \lambda_3(\mu_2) = \Lambda_1^1 > 0.
\]
Moreover, since we are in the situation 
\[
 \left\{
    \begin{array}{ll}
       s_1>0>r_1 \\
       r_2>0>s_2
    \end{array},
\right.
\]
the assumption $ \Lambda_0^1 = r_1+r_2<0 $ 
  and $    \Lambda_0^2 = s_1+s_2<0$
is equivalent to   $-r_1>r_2>0$ and $ -s_2>s_1>0$, which implies $r_1s_2>r_2s_1$. Hence, by Proposition \ref{prop : impossiblepersis} below, $\lambda_3(\nu_3) > 0$ which concludes the proof.
 
\end{proof}

\begin{ex}
Take $s_1 = 1/3$, $r_1 = - 1/3$, $s_2=-3/8$ and $r_2 = 1/4$. Choose $q_1 = q_2$ so that $p_1 = p_2 = \frac{1}{2}$. Then
\[
\Lambda_0^1 = - \frac{1}{24}; \quad \Lambda_0^2 = - \frac{1}{48}; \quad \Lambda_0^3 = - \frac{1}{5},
\]
thus conclusion of Theorem \ref{thm:neutralinvade} holds. Illustrations are given in Figure \ref{fig:neutraltriangle}.

\begin{figure}[htbp]
\begin{minipage}[c]{.45\linewidth}
\begin{center}
\includegraphics[scale=0.35]{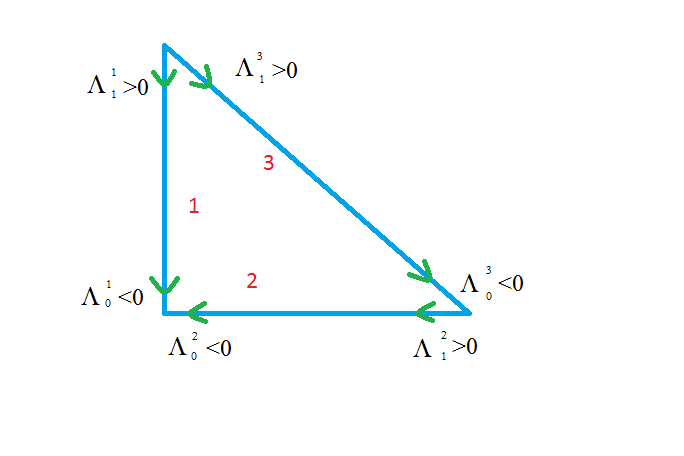}
\caption{behavior of the process on a neighbourhood of each vertex.}
\label{fig:neutraltriangle}
\end{center}
\end{minipage}
\hfill
\begin{minipage}[c]{.45\linewidth}
\begin{center}
\includegraphics[scale=0.35]{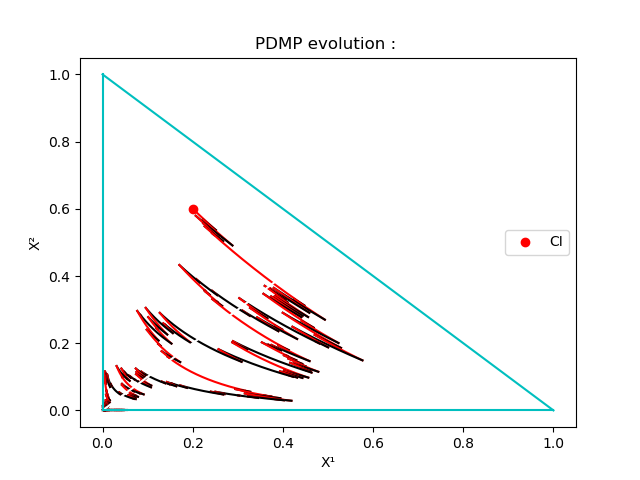}
\caption{Neutral species invade the community. Environment values $[0.5,-0.7]$ and $[-0.8,0.2]$.}
\label{fig:neutral}
\end{center}
\end{minipage}
\end{figure}

\end{ex}
An immediate corollary of Theorem \ref{thm:suffinvade} is the following, which gives a sufficient condition for non invasion of species 3. 
\begin{cor}
Assume that $\Lambda_0^1>0$ or $\Lambda_0^2>0$. For $\varepsilon>0$, define
\[\tau^{\varepsilon} = \inf\{ t \geq 0 \: \|(X_t,Y_t) \| \geq \varepsilon\}.
\]
Then, there exist $\varepsilon>0$, $b >1$, $\theta >0$ and $c>0$ such that, for all $(x,y) \in \mathscr{E} \setminus \mathscr{E}_0$, and all $\mathbf{s} \in E$,
\[
\mathbb{E}_{(x,y,\mathbf{s})} ( e^{ \tau^{\varepsilon}} ) \leq c ( 1 + \|(x,y\|^{-\theta}).
\]
In particular, species $3$ cannot exclude the two other species.
\end{cor}

\subsubsection{Persistence is impossible for three species and two environments}

We show in this section that persistence is not possible with only two environments. The following section will give an example of persistence for 3 environments.

The following lemma, which is a consequence of results in \cite{benaimpersistence}, \cite{BL16} and \cite{BS17},  ensures us that if an edge has an attractive index, then the process converges to this face, thus preventing persistence of the process.

\begin{lemma}
\label{lemma : extinct}
Assume that for some species $i$, $\nu_i$ exists and that $\lambda_i(\nu_i) < 0$. Then, for all $(x,y,s) \in M \setminus M_0^i$, 
\[
\pp_{\big((x,y),s\big)} \left( \limsup_{t \to \infty} \frac{1}{t} \log ( X_t^i ) \leq \lambda_i(\nu_i) \right) = 1.
\]
In particular, species $i$ goes to extinction and persistence is not possible in that case.
\end{lemma}

The strategy to prove the lack of persistence is the following. First, we know by Theorem \ref{thm:suffinvade} that when one vertex is attractive (i.e satisfies $\Lambda_i < 0$), the corresponding species has a positive probability to invade the community, hence preventing  the persistence of the process. Thus, we will assume that all the vertex are repulsive. Then, we show that under this assumption, there is always, in the configuration \eqref{ordrefitness}, at least one non-trivial ergodic measure (supported by an edge). Finally, we prove that among all the ergodic measure   $\nu_1$, $\nu_2$, $\nu_3$, there is a species with a negative invasion rate with respect to this measure, which shows that this species has a positive probability to disappear and thus that persistence is not possible. 

\begin{prop}
Assume that all the vertex of the triangle are repulsive, i.e for each $i\in \{1,2,3\}$, $\Lambda_i > 0.$ Then  necessarily  there exists an invariant measure on a side of a triangle.
\end{prop}

\begin{proof}

We prove this result by contradiction. 
Refers to proposition \ref{prop:oneD}, if on the edge $i$, $\Lambda_0^i$ and $\Lambda_1^i$ are positive, there exists an invariant measure on this edge.
Thus the only configuration possible to obtain no invariant measure on the edges, when the vertex are repulsive is the one given in Figure \ref{fig:neutraltriangle} (or the symmetrical case).
\begin{figure}[htbp]
\includegraphics[scale=0.35]{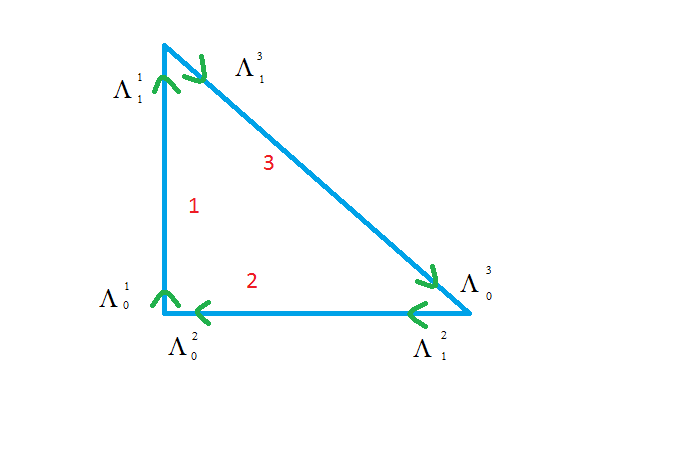}
\caption{Impossible situation.}
\label{fig:neutraltriangle}
\end{figure}
On the one hand, one has, with the configuration of figure \ref{fig:neutraltriangle},  the following inequalities :

   \[ 
\left\{
\begin{aligned}
 \Lambda_{1}^1  & = & -p_1\frac{r_1}{1+r_1}-p_2\frac{r_2}{1+r_2}<0 \\
      \Lambda^3_{1}   & = & p_1\frac{s_1-r_1}{1+r_1}+p_2\frac{s_2-r_2}{1+r_2}>0
\end{aligned}
\right.
\iff
\left\{
\begin{aligned}
  & -p_1\frac{r_1}{1+r_1}<p_2\frac{r_2}{1+r_2} \\
  & p_1\frac{s_1}{1+r_1}>-p_2\frac{s_2}{1+r_2}
\end{aligned}
\right.
\]
and  thus

$$ r_1s_2<r_2s_1. $$
But on the other hand, we have:

   \[ 
\left\{
\begin{aligned}
 \Lambda_{1}^2  & = & -p_1\frac{s_1}{1+s_1}-p_2\frac{s_2}{1+s_2}>0 \\
      \Lambda^3_{0}   & = & p_1\frac{r_1-s_1}{1+s_1}+p_2\frac{r_2-s_2}{1+s_2}<0
\end{aligned}
\right.
\iff
\left\{
\begin{aligned}
  & -p_1\frac{r_1}{1+s_1}>p_2\frac{r_2}{1+s_2} \\
  & p_1\frac{s_1}{1+s_1}<-p_2\frac{s_2}{1+s_2}
\end{aligned}
\right.
\]
which imply

$$ r_1s_2>r_2s_1, $$
a contradiction.

\end{proof}

The previous proposition states that there exists at least one ergodic measure on a edge of the triangle. The following proposition deals with the case where there is only one such measure :

\begin{prop}
Assume it exists exactly one invariant measure on a edge of a triangle, then a species has an negative invasion rate with respect to this measure. And this species goes to extinction.
\end{prop}

\begin{proof}
Two case have to be considered :
\begin{enumerate}
    \item The invariant measure is on the face 3.
    \item The invariant measure is on the face 1 (or symmetrically on the face 2).

    Let us consider first the case 1, we are in the situation of the figure \ref{fig:prop21} ( or in a symmetrical case): 
    
\begin{figure}[htbp]
\includegraphics[scale=0.35]{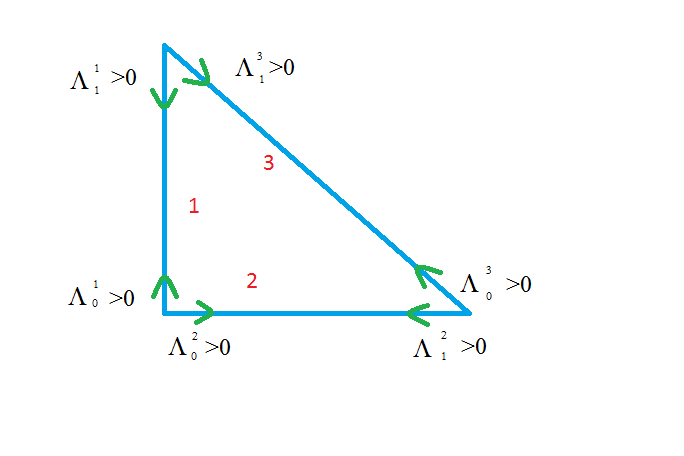}
\caption{Ergodic measure on each edge.}
\label{fig:prop21}
\end{figure}
\end{enumerate}

So the following inequalities hold : 

 \[ 
\left\{
\begin{aligned}
 \Lambda_{1}^1  & =-p_1\frac{r_1}{1+r_1}-p_2\frac{r_2}{1+r_2}<0 \\
      \Lambda^3_{1}&=p_1\frac{s_1-r_1}{1+r_1}+p_2\frac{s_2-r_2}{1+r_2}>0
\end{aligned}
\right.
\iff
\left\{
\begin{aligned}
  & -p_1\frac{r_1}{1+r_1}<p_2\frac{r_2}{1+r_2} \\
  & p_1\frac{s_1}{1+r_1}>-p_2\frac{s_2}{1+r_2}
\end{aligned}
\right.
\]
 \quad\\
we obtain  $ r_1s_2<r_2s_1 $, and so $\lambda_3(\nu_3)<0$.\\
By lemma \ref{lemma : extinct}, species 3 goes to extinction.

For the case 2, the same reasoning is still valid. Species 1 goes to extinction.

\end{proof}

The following property  clarifies the conditions to obtain  negative invasion rates. Moreover it states, it's impossible to obtain  three positive invasion rates $\lambda_i(\nu_i)$ for $i\in\{1,2,3\}$ : 
\begin{prop}
\label{prop : impossiblepersis}
The following equivalences hold:
$$\lambda_1(\nu_1)<0 \Leftrightarrow \lambda_2(\nu_2)<0 \Leftrightarrow \lambda_3(\nu_3)>0 \Leftrightarrow s_1r_2<s_2r_1$$

\end{prop}
\begin{proof}
Remember, we are in situation \ref{fitdefa}, it may possible to simplify the following inequalities :
\begin{align*}
    \lambda_1(\nu_1)&<0 \iff  \frac{s_1}{|r_1|} + \frac{s_2}{|r_2|}<0 \iff  r_1s_2>s_1r_2\\
    \lambda_2(\nu_2)&<0 \iff  \frac{r_1}{|s_1|} + \frac{r_2}{|s_2|} \iff  r_1s_2>s_1r_2\\
    \lambda_3(\nu_3)&>0 \iff -\frac{r_1}{|s_1-r_1|} - \frac{r_2}{|s_2-r_2|}>0 \iff  r_1s_2>s_1r_2\\
\end{align*}
And we find the result.
\end{proof}

The previous property  states it impossible have persistence if we have two invariant measures on the faces 1 and 3 or 2 and 3. Because necessarily one of the invasion rates compared to one of these two measure is negative

\begin{prop}
If there exist two invariant measures on face 1 and 2, necessarily it exists one on face 3 and $\lambda_3(\nu_3)<0$.
\end{prop}

\begin{proof}
Assume it exists an invariant measure on edge 1 and 2. Let's prove $\Lambda_0^3$ and $\Lambda_1^3$ are strictly positive.
We are in the situation of   figure \ref{fig:prop23}

\begin{figure}[htbp]
\includegraphics[scale=0.35]{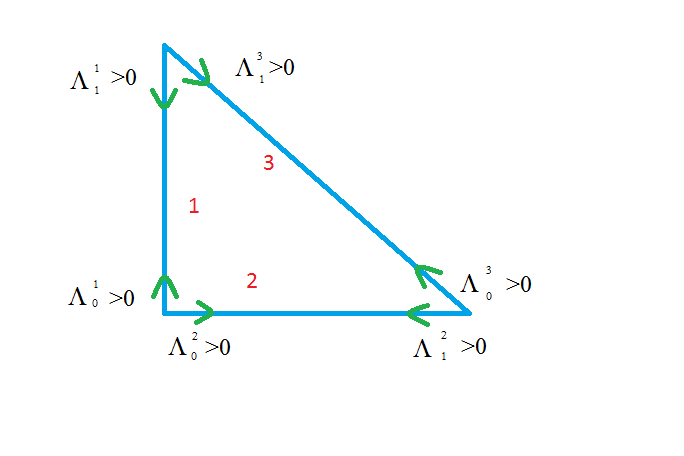}
\caption{behavior of the process on a neighbourhood of each vertex.}
\label{fig:prop23}
\end{figure}
 Then we have  
   \[ 
\left\{
\begin{aligned}
 \Lambda_{1}^2  & =  -p_1\frac{s_1}{1+s_1}-p_2\frac{s_2}{1+s_2}>0 \\
      \Lambda^0_{1}   & = r_1+p_1+r_2p_2>0
\end{aligned}
\right.
\]
   
 and \begin{align*}
     \Lambda_0^3&=p_1\frac{r_1-s_1}{1+s_1}+p_2\frac{r_2-s_2}{1+s_2}\\
     &=-p_1\frac{s_1}{1+s_1}-p_2\frac{s_2}{1+s_2}+ \frac{p_1r_1+p_2r_2+p_1r_1s_2+p_2s_1r_1}{(1+s_1)(1+s_2)}\\
     &=  \Lambda_{1}^2 + \frac{\Lambda_0^1+p_1r_1s_2+p_2s_1r_1}{(1+s_1)(1+s_2)}
 \end{align*}

Note that $p_1r_1s_2$ and $p_2s_1r_1$ are positive because $r_1$, $s_2$ are both negative and $s_1$, $r_2$ are both positive. So $\Lambda_0^3$ is positive.

A similar reasoning gives the same condition for  $\Lambda_1^3$.\\

Let us conclude the proof by noting that

 \[ 
\left\{
\begin{aligned}
 \Lambda_{0}^2  & =p_1s_1+p_2s_2>0 \\
      \Lambda^0_{1}   & = r_1+p_1+r_2p_2>0
\end{aligned}
\right.
\iff
s_1r_2>s_2r_1 \iff \lambda_3(\nu_3)<0
\]

And species 3 goes to extinction.

\end{proof}

\begin{cor}
If there are exactly two or three invariant measures, at least, an invasion rate $\lambda_i(\nu_i)$, is negative.
\end{cor}

\begin{proof}
The proof is direct with property \ref{prop : impossiblepersis}
\end{proof}

It is impossible to obtain persistence with only two environments and three species. In the next section we give an example of persistence with three species and three environments.

\subsection{Persistence for 3 species and 3 environments}
 
In this part we  give two examples of a system with 3 species and 3 environments which is persistent. In the first example, each species has, in one environment, the best fitness. In the second, the last species has never the best fitness. 

With obvious notations, we set
\begin{align*}
    &\Lambda_{0}^1=p_1 r_1+p_2 r_2+p_3r_3,\quad \Lambda_{1}^1=-p_1\frac{r_1}{1+r_1}-p_2\frac{r_2}{1+r_2} -p_3\frac{r_3}{1+r_3} \\
    &\Lambda_{0}^2=p_1s_1+p_2s_2+p_3s_3 ,\quad \Lambda_{1}^2=-p_1\frac{s_1}{1+s^1}-p_2\frac{s_2}{1+s_2} -p_3\frac{s_3}{1+s_3} \\
    & \Lambda_{0}^3=p_1\frac{r_1-s_1}{1+s_1}+p_2\frac{r_2-s_2}{1+s_2}+p_3\frac{r_3-s_3}{1+s_3}, \\ &\Lambda^3_{1}=p_1\frac{s_1-r_1}{1+r_1}+p_2\frac{s_2-r_2}{1+r_2}+p_3\frac{s_3-r_3}{1+r_3}.
\end{align*}

To prove persistence, with theorem \ref{thm:persistence}, we need to calculate   the invasion rates $\lambda_i(\nu_i)$ for possible ergodic measures $\nu_i$ (remember if $k\neq i$, $ \lambda_k(\nu_i)=0$). Thus, we need to obtain an explicit formula for the $\nu_i$ density. We could follow the same reasoning as in lemma \ref{lem:density}.
If $(U_t)_{t>0}$ admits an invariant measure  on $(0,1)\times E$, $\Pi$, it is absolutely continuous with respect to the Legesgue measure and admits a density. We still define $h_i$ the density of $\Pi(.,i) $ for $i\in \{1,2,3\}$. The $h_i$ are in $\mathscr{C}^{\infty}$ and verifie the Fokker-Planck equation  :
\begin{equation}
    \begin{cases}
     -h_1(q_{1,2}+q_{1,3}) + q_{2,1} h_2+q_{3,1}h_3 & = -( g_1 h_1)'\\
   -h_2(q_{2,1}+q_{2,3}) + q_{1,2} h_1+q_{3,2}h_3 & = ( g_2 h_2)'\\
    -h_3(q_{3,2}+q_{3,1}) + q_{1,3} h_1+q_{2,3}h_2 & = ( g_3 h_2)'
    \end{cases}
    \label{eq:fokker2}
\end{equation}

Unfortunately, the explicit computation of $(h_i)$ now becomes tedious, and thus we have no explicit expression for the invasion rates $\lambda_i(\nu_i)$. 
In  the following example, we compute numerically the invasions rates $\lambda_i(\nu_i)$ and show that we are in the situation of Theorem \ref{thm:persistence}.

\begin{ex}
\label{ex:pdmppersi3env1}
In this example, we assume  $q_{i,j}=q_{k,l},\forall i,j,k,l \in \{1,2,3\} $ , and $\mathbf{s}_1=\{1,\frac{1}{2}\}$, $\mathbf{s}_2=\{\frac{-1}{4},\frac{-1}{2}\}$, $\mathbf{s}_3=\{\frac{-1}{3},\frac{1}{3}\}$.

And it follows

\begin{align*}
    &\Lambda_{0}^1=\frac{1}{6},\quad \Lambda_{1}^1=\frac{5}{36},\quad
    \Lambda_{0}^2=\frac{5}{36},\quad \Lambda_{1}^2=\frac{1}{6},\quad
     \Lambda_{0}^3=\frac{5}{36},\quad \Lambda^3_{1}=\frac{1}{6}.
\end{align*}

Then there are exactly 3 ergodic measures $\nu_i, i\in \{1,2,3\}$, so we are  in the same situation than picture \ref{fig:prop23}.\\

For each of them, we  approximate the invasion rate, i.e the quantity $\lambda_i(\nu_i)=\int_{\mathscr{E}_0^i\times E} F_{s}^i(x)d\nu_i(x)$.
A way of doing it, since the measure is ergodic, is to use the ergodic theorem. So if $(X_t,s_t)_{t>0}$ is the PDMP define by equation \ref{eqprincipalepdmp} and starting in $\mathrm{int}(\mathscr{E}_0^i)\times E$, $\frac{1}{T}\int_0^TF^i_{s_t}(X_t)dt\underset{T\to +\infty}{\longrightarrow}\lambda_i(\nu_i)$.\\  
Note the initial point $(X_0,s_0)$ has no influence on the previous result. In the following we chose  arbitrary initial conditions, others would have led to the same results.\\
We proceed as follows: \begin{itemize}
    \item Simulate a large number of PDMP trajectories (1000), on $[0,T]$ for $T$ big enough and starting in $\mathrm{int}(\mathscr{E}_0^i)\times E$.
    \item For each simulation calculate $\frac{1}{T}\int_0^TF^i_{s_t}(X_t)dt$.
    \item Take the average on the trajectories to improve our result. 
\end{itemize}

Results :\\
In each simulation we take $T=80$ and the number of trajectories simulated is 1000.
\begin{enumerate}
    \item For $X_0=[0,0.5],s_0=\mathbf{s}_1$ we obtain $\lambda_1(\nu_1)=0.0191$
    \item For $X_0=[0.5,0],s_0=\mathbf{s}_1$ we obtain $\lambda_2(\nu_2)=0.0594$
\item For $X_0=[0.5,0.5],s_0=\mathbf{s}_1$ we obtain $\lambda_3(\nu_3)=0.090$
\end{enumerate}

For this configuration of environments we obtain positive invasion rates. So  it proves, using \ref{thm:persistence}, that persistence is possible for 3 species and 3 environments.

\begin{figure}[htbp]
\begin{minipage}[c]{.45\linewidth}
\begin{center}
\includegraphics[scale=0.35]{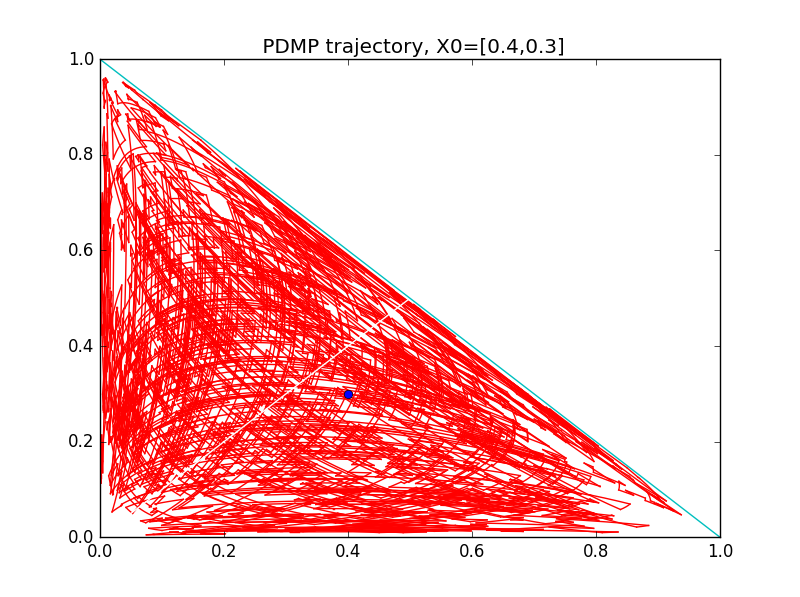}

\label{fig:pdmppers1}
\end{center}
\end{minipage}
\hfill
\begin{minipage}[c]{.45\linewidth}
\begin{center}
\includegraphics[scale=0.35]{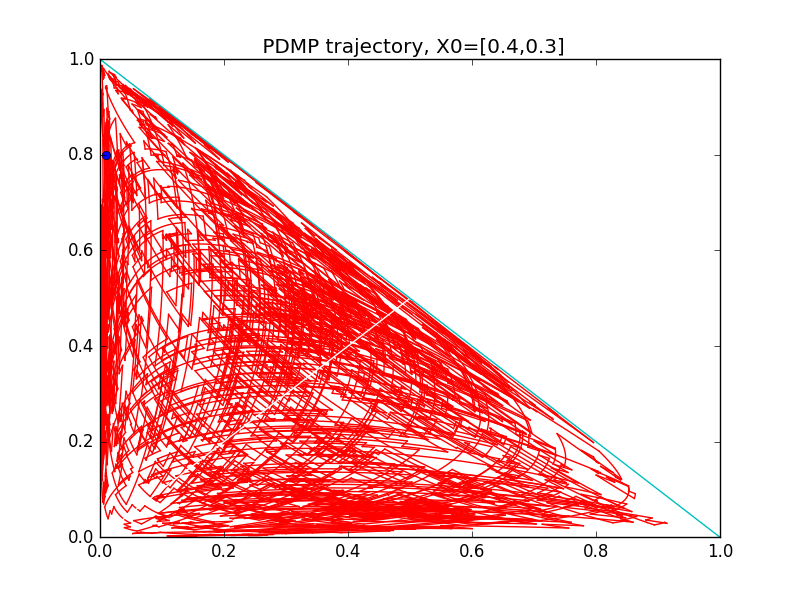}

\label{fig:pdmppers2}
\end{center}
\end{minipage}
\caption{A PDMP trajectory. $ T=1000$. Left figure: $ X_0=[0.4,0.3]$.     Right figure: $X_0=[0.01,0.8]$}
\end{figure}

\end{ex}

\begin{ex}
\label{ex:pdmppersi3env2}

We now give an example of persistence in which a species has never the best fitness. Let us choose as parameters $q_{i,j}=q_{k,l},\forall i,j,k,l \in \{1,2,3\} $ , and $\mathbf{s}_1=\{0.1,-0.3\}$, $\mathbf{s}_2=\{-0.33,0.1\}$, $\mathbf{s}_3=\{0.27,0.25\}$.

So that,

\begin{align*}
    &\Lambda_{0}^1=0.5,\quad \Lambda_{1}^1=0.13,\quad
    \Lambda_{0}^2=0.04,\quad \Lambda_{1}^2=0.19,\quad
     \Lambda_{0}^3=0.26,\quad \Lambda^3_{1}=0.20.
\end{align*}

As in the previous example,  there are exactly 3 ergodic measures $\nu_i, i\in \{1,2,3\}$ and  we are  in the same situation than Figure \ref{fig:prop23}.\\
Let us calculate the invasion rate, with the same method :

\begin{enumerate}
    \item For $X_0=[0,0.5],s_0=\mathbf{s}_1$ we obtain $\lambda_1(\nu_1)=0.016$
    \item For $X_0=[0.5,0],s_0=\mathbf{s}_1$ we obtain $\lambda_2(\nu_2)=0.019$
\item For $X_0=[0.5,0.5],s_0=\mathbf{s}_1$ we obtain $\lambda_3(\nu_3)=0.009$
\end{enumerate}

The invasion rates are strictly positive so we conclude that in this example; we have persistence, even if the last species has never the best fitness.
A trajectory of the PDMP is plotted in Figure \ref{fig:persist2}.

\begin{figure}
    \centering
    \includegraphics[scale=0.35]{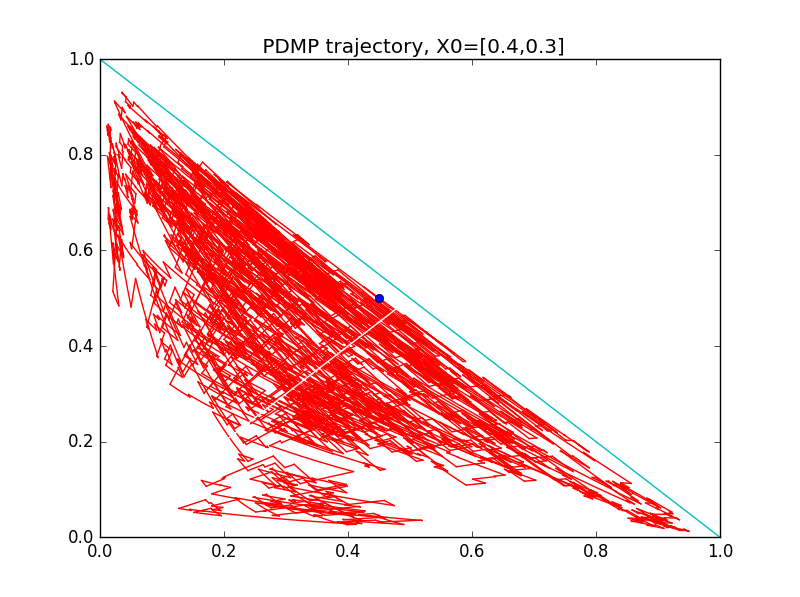}
    \caption{PDMP trajectory, T=1000, 
    $X_0=[0.45,0.5]$. 
    }
    \label{fig:persist2}
\end{figure}

\end{ex}
In the two previous examples, the process is persistant, and thus admits at least one stationnary distribution $\Pi$ such that $\Pi(M_+)= 1$. The numerical simulations presented in Figure  \ref{fig:pdmppers2} suggest that in the case of Example \ref{ex:pdmppersi3env1}, that $\Pi$ has full support, i.e. its support is the whole space $M$. The following proposition proves this fact, as well as the exponential convergence of the process towards $\Pi$ for general coefficients.

\begin{prop}
Assume that $\mathbf{s}_1$, $\mathbf{s}_2$, $\mathbf{s}_3$ are such that 
\begin{equation}
  \begin{cases}
  s_1 > 0 > r_1\\
  r_2 > s_2 > 0\\
  0 > r_3 > s_3
    \end{cases}
    \label{eq:fullsupport1}
\end{equation}
or
\begin{equation}
  \begin{cases}
  s_1 > r_1 > 0\\
  r_2 > 0 > s_2\\
  0 > s_3 > r_3
    \end{cases}
    \label{eq:fullsupport2}
\end{equation}
Then, if the process is persistent :
\begin{enumerate}
    \item The stationary distribution $\Pi$ satisfying $\Pi(M_+)=1$ is unique and absolutely continuous with respect to the Lebesgue measure;
    \item $\Pi$ has full support;
    \item If furthermore $\alpha:= s_2r_3 - r_2 s_3 + s_3 r_1 - s_1 r_3 + s_1r_2 - s_2r_1 \neq 0$, then there exist $C, \theta, \gamma > 0$ such that, for all $(x,y,s) \in M_+$ and all $t \geq 0$,
    \[
    \| \mathbb{P}_{(x,y,s)}( U_t \in \cdot ) - \Pi \|_{TV} \leq C \left( \frac{1}{x^{\theta}} + \frac{1}{y^{\theta}} + \frac{1}{(1-x-y)^{\theta}} \right) e^{ - \gamma t }.
    \]
\end{enumerate} 
\end{prop}

\begin{proof}
We only do the proof in the case \eqref{eq:fullsupport1}; it can be proved in the same way for \eqref{eq:fullsupport2}. 

For $i \in \{1,2,3\}$, we denote by $(\varphi_t^i(x))_{t \geq 0}$ the flow generated by $G_i:=G_{s_i}$ and started at $x$. That is, $\varphi_t^i(x)$ is solution to
\[
\begin{cases}
\frac{\partial \varphi_t^i(x)}{\partial t} = G_i(\varphi_t^s(x))\\
\varphi_0^i(x) = x.
\end{cases}
\]
 For $m \geq 1$,  $\mathbf{i}=(i_1, \ldots, i_m) \in \{1,2,3\}^m$ and $\mathbf{u} = (u_1,\ldots,u_m) \in \mathbb{R}_+^m$, we denote by $\mathbf{\Phi}_{\mathbf{u}}^{\mathbf{i}}$ the composite flow: $\mathbf{\Phi}_{\mathbf{u}}^{\mathbf{i}} = \varphi_{u_m}^{i_m} \circ \ldots \circ \varphi_{u_1}^{i_1}$.  For $x \in M$ and $t \geq 0$, we denote by $\mathcal{O}^+_t(x)$ (resp. $\mathcal{O}^+(x)$)  the set of points that are reachable from $x$ at time $t$ (resp. at any nonnegative time) with a composite flow: $$\mathcal{O}^+_t(x)=\{ \mathbf{\Phi}_{\mathbf{v}}^{\mathbf{i}} (x), \: (\mathbf{i},\mathbf{v}) \in E^m \times \mathbb{R}_+^m, m \in \mathbb{N}, v_1 + \ldots + v_m = t\},$$

$$ 
\mathcal{O}^+(x) = \bigcup_{t \geq 0} \mathcal{O}^+_t(x).
$$ 
We define the set of points that are accessible from $\mathscr{E}_+$ as 
\[
\Gamma = \bigcap_{ x \in  \mathscr{E}_+} \overline{\mathcal{O}^+(x)}.
\]
According to Corollary 4.6 and Remark 14 in \cite{benaimpersistence} and to Theorem 3.4 and Proposition 3.11 in \cite{BMZIHP}, $\Gamma \times E$ is included in the support of $\Pi$ for all stationary distribution $\Pi$ such that $\Pi(M_+) = 1$. Thus, point (2) is proved if we show that $\Gamma = \mathscr{E}$.

For $x \in \mathscr{E}$ and $i \in \{1,2,3\}$, we let $\gamma_+^i(x)$ and $\gamma_-^i(x)$ denote the positive and negative orbit, respectively,  of $x$ under the flow $\varphi^i$, namely :
\[
\gamma_+^i(x) = \{ \varphi_t^i(x), t \geq 0 \}, \quad \gamma_-^i(x) = \{ \varphi_t^i(x), t \leq 0 \}.
\]
In situation \eqref{eq:fullsupport1}, for all $x \in \mathscr{E}_+$, $\gamma_+^i(x)$ is a regular curve linking $x$ to $e_i$, while $\gamma_-^i(x)$ is a curve linking $x$ to $e_{i+1}$, where $i+1$ is taken modulo $3$. We claim that for all $x,y \in \mathscr{E}_+$ with $x \neq y$, there exist $i,j \in \{1,2,3\}$ such that $\gamma_+^i(x) \cap \gamma_-^j(y) \neq \emptyset$. In that case, let $z \in \gamma_+^i(x) \cap \gamma_-^j(y) \neq \emptyset$. Then, there exist $t,s > 0$ such that $z = \varphi_t^i(x) = \varphi_{-s}^j(x)$. Hence, by the flow property of $\varphi^j$, we get $y = \varphi^j_s \circ \varphi_t^i(x)$, and thus $y$ is accessible from $x$. This shows that $\mathscr{E}_+ \subset \mathcal{O}^+(x)$, for all $x \in  \mathscr{E}_+$, hence $\mathscr{E} \subset \Gamma$ and point (2) is proven since $\mathscr{E} \supset \Gamma$. We now prove the claim. Let $x,y \in \mathscr{E}_+$ with $x \neq y$. Then the $\gamma_+^i(x), i \in \{1,2,3\}$ are three regular curves linking $x$ to each vertex of $\mathscr{E}_0$, while  the $\gamma_-^j(x), i \in \{1,2,3\}$ are three regular curves linking $y$ to each vertex of $\mathscr{E}_0$. In particular, at least one of these three first curves has to cross one of the three other, which proves the claim.

We pass to the proof of point (1). Let $x \in \mathscr{E}_+$. Then, there must exists a point $y \in \gamma_+^1(x)$ such that $G_1(y)$ and $G_2(y)$ are linearly independent. If not, one would have for all $y \in \gamma_+^1(x)$, $G_1(y) = \alpha(y) G_2(y)$ for some negative function $\alpha$, which would imply that $\gamma_+^1(x) = \gamma_-^2(x)$. This is a contradiction since $\gamma_+^1(x)$ is a curve linking $x$ to $e_1$ while $\gamma_-^2(x)$ is a curve linking $x$ to $e_3$. Hence, the weak bracket condition holds at $y$ which belongs to $\Gamma$, thus $\Pi$ has to be unique and absolutely continuous with respect to the Lebesgue measure by Corollary 6.3 in \cite{benaimpersistence}.

To prove point (3), we also use Corollary 6.3 in \cite{benaimpersistence}. According to this corollary (which is a consequence of results in \cite{BMZIHP} and \cite{BH12}), it suffices to find an accessible point $(x,y)$ (thus a point in $\mathscr{E}$) such that $G_1(x,y) - G_2(x,y)$ and $G_2(x,y) - G_3(x,y)$ are lineary independent. One can check that 
\[
\det \left( G_1 - G_2, G_2 - G_3 \right)(x,y) = \frac{x y (1 - x -y)}{h(x,y)}(s_2r_3 - r_2 s_3 + s_3 r_1 - s_1 r_3 + s_1r_2 - s_2r_1),
\]
where $h(x,y) = (1 + s_1 x + r_1 y)(1 + s_2 x + r_2 y)(1 + s_3 x + r_3 y)>0$. In particular, if $s_2r_3 - r_2 s_3 + s_3 r_1 - s_1 r_3 + s_1r_2 - s_2r_1 \neq 0$; $G_1(x,y) - G_2(x,y)$ and $G_2(x,y) - G_3(x,y)$ are lineary independent and point (3) is proven.
\end{proof}

\begin{remarque}
Using the same kind of proof, it can be shown that $\Pi$ has full support in more general situation. Indeed, the assumption that the process is persistent implies that each of the species has exactly once the worst fitness. Indeed, if one species (say 1) has never the worst fitness, it has to be better than another one (say 2) in each environment, which implies that species 2 will go to extinction according to Corollary \ref{droiteextinctioncorr}. Moreover, if the process is persistent, by Theorem \ref{th:alwaysthebest}, none of the species can have the best fitness in every environment. Thus, in the persistent case, we are either in case \ref{eq:fullsupport1} or \ref{eq:fullsupport2}, that are handled by the previous proposition, or in a situation like in Example \ref{ex:pdmppersi3env2}. That is
\begin{equation}
  \begin{cases}
  s_1 > 0 > r_1\\
  r_2 > 0 > s_2\\
  s_3 > r_3 > 0
    \end{cases}
    \label{eq:fullsupport3}
\end{equation}
With the same notations as in the proof of the previous proposition, one can see that the $\gamma_-^i(y)$ for $i \in \{1,2,3\}$ are still three regular curves linking one to each vertex of $\mathscr{E}_0$. However, now $\gamma_+^1(x)$ and $\gamma_+^3(x)$ are curves linking $x$ to $e_1$, while $\gamma_+^2(x)$ is a curve linking $x$ to $e_2$. To be able to conclude as in the previous proof, we need $e_3=0$ to be accessible from $x$. This is for example the case if one can find $a_1,a_2,a_3 \geq 0$ such that 
\begin{equation}
    a_1 s_1 + a_2 s_2 + a_3 s_3 < 0 \quad \mbox{and} \quad a_1 r_1 + a_2 r_2 + a_3 r_3 < 0, 
    \label{eq:0access}
\end{equation}
by the argument given in the proof of Theorem \ref{thm:neutralinvade}. In this case, $e_2 \in \mathcal{O}^+(x)$ and as before, $\mathcal{0}^+(x)$ has to have a nonempty intersection with one of the $\gamma_-^i(y)$, and we can conclude that in this situation, $\Pi$ is unique and has full support. In Example \ref{ex:pdmppersi3env2}, take $a_1=a_2=1/2$ and $a_3 =0$, then \eqref{eq:0access} is satisfied and thus $\Pi$ is unique and has full support. Note also that the condition $ s_2r_3 - r_2 s_3 + s_3 r_1 - s_1 r_3 + s_1r_2 - s_2r_1 \neq 0$ is satisfied,  thus the process converges exponentially fast to $\Pi$.
\end{remarque}

%

%


\subsection{Conclusion}
In this part, we treat the whole case of tree species and two environments, and we prove persistence is impossible in this situation. On the other hand, if we deal with 3 species and 3 environments, we are able to exhibit configurations where persistence is possible.
Of course many other configurations of fitness gives persistence to.
But, we actually are unable to give explicit conditions on fitness to  characterize them. However numerical (deterministic) approximations may be used for each environment.\\

In view of the results of the previous sections, we formulate the following conjecture, that we have proved for 2 and 3 species.
\begin{conj}
Persistence of the system is possible if and only if there are at least as many environments as there are species.
\end{conj}

\section{Appendix}
 
\subsection{Proof of Theorem \ref{thqttapprox}}
Let $(U_n^J )_{n \geq 0}= (X_n^J, s_n^J)_{n \geq 0}$ be the Moran model described in section \ref{presentation}, $(U_t)_{t \geq 0} = (X_t, s_t)_{t \geq 0}$ the PDMP given by equation \ref{eqprincipalepdmp} and $(Z_t^J)_{t \geq 0} = (X_{t/J}, s_{t/J})_{t \geq 0}$ the PDMP in the time scale of the discrete process, whose generator is given by $L$. 

Our goal is to quantify the difference
\[
\mathbb{E}_{x,s} \left( f(U_t) - f( \tilde{U}^J_t) \right) 
\]
for $f : \mathscr{I} \to \mathbb{R}$ regular enough. First note that, since $U_t = Z_{tJ}^J$,  we can rewrite this difference as 
\[
\mathbb{E}_{x,s} \left( f(Z_{tJ}^J) - f(Z_{ \lfloor t J \rfloor}^J) \right) +  \mathbb{E}_{x,s} \left( f(Z^J_{ \lfloor t J \rfloor}) - f(U^J_{\lfloor t J \rfloor} \right) + \mathbb{E}_{x,s} \left( f(U^J_{\lfloor t J \rfloor} - f(\tilde{U}^J_t) \right)  . 
\]
We first show that the first and last term of the above quantity can easily be controlled. For all $t \geq 0$ and $J \geq 1$, one has $ 0 \leq t - \frac{\lfloor t J \rfloor}{J} < \frac{1}{J}$, hence $\frac{\lfloor t J \rfloor}{J}$ converges to $t$. The probability that $(s_u)_{u \geq 0}$ perfomes a jump at time $t$ is zero, hence, almost surely, for $J$ big enough, $s_\frac{\lfloor t J \rfloor}{J} = s_t$. In particular, if $f$ is $C^1$, one has
\[
\left| f(Z_{tJ}) - f(Z_{\lfloor t J \rfloor}) \right| \leq \|f^{(1)}\|  \left| X_{t} - X_{\frac{\lfloor t J \rfloor}{J}} \right|,
\]
where $\|f^{(1)}\| = \max_{(x,s) \in \mathscr{I}} | f'_s(x) |$. Let $C_G = \max_{(x,s) \in \mathscr{I}} |G_s(x)|$, then
\begin{align*}
    \left| X_{t} - X_{\frac{\lfloor t J \rfloor}{J}} \right| & = \left| \int_{\frac{\lfloor t J \rfloor}{J}}^t G_{s_u}(X_u) du \right|\\
    & \leq \left( t - \frac{\lfloor t J \rfloor}{J} \right) C_G \leq \frac{C_G}{J},
\end{align*}
which implies 
\[
\left| f(Z_{tJ}) - f(Z_{\lfloor t J \rfloor}) \right| \leq \|f^{(1)}\| \frac{C_G}{J}.
\]
Furthermore, $\tilde{U}_t^J=(\tilde{X}_t^J, \tilde s_t^J)$ and by definition, $\tilde s_t^J = s_{ \lfloor t J \rfloor}^J$ and
\begin{align*}
    \left| \tilde{X}_t^J - X_{ \lfloor t J \rfloor}^J \right| & = J \left( t - \frac{\lfloor t J \rfloor}{J} \right) \left|X_{ \lfloor t J \rfloor + 1}^J - X_{ \lfloor t J \rfloor}^J  \right|\\
    & \leq \left| X_{ \lfloor t J \rfloor + 1}^J - X_{ \lfloor t J \rfloor}^J  \right|.
\end{align*}
Thus,
\[
\left| \mathbb{E}_{x,s} \left( f(U^J_{\lfloor t J \rfloor} - f(\tilde{U}^J_t) \right)\right| \leq \|f^{(1)}\|\mathbb{E}_{x,s} \left| X_{ \lfloor t J \rfloor + 1}^J - X_{ \lfloor t J \rfloor}^J\right| \leq \frac{\|f^{(1)}\|}{J} 
\]
because the difference between $X^J_{n+1}$ and $X^J_n$ is at most $1/J$. 
Therefore, one can look at the difference 
\[
\mathbb{E}_{x,s} \left( f(Z_{\lfloor t J \rfloor}^J) - f(U_{ \lfloor t J \rfloor}^J) \right) 
\]
because 
\[
\left| \mathbb{E}_{x,s} \left( f(U_t) - f(U_{ \lfloor t J \rfloor}^J) \right) \right| \leq  \|f^{(1)}\|\frac{C_G+1}{J} + \mathbb{E}_{x,s} \left( f(Z_{\lfloor t J \rfloor}^J) - f(U_{ \lfloor t J \rfloor}^J) \right).
\]
 We  set $T_u^J f(x,s) = \mathbb{E}_{x,s}\left( f(Z_t^J) \right)$ and $S_k^J f(x,s) = \mathbb{E}_{x,s}\left( f(X_n^J) \right)$. When there is no ambiguity, we shall drop the exponent $J$ on $T^J$ and $S^J$. With these notations, our aim is to control $\| T_n f - S_n f \|_J$, where $n =\lfloor t J \rfloor $.For this, we can use the following inequality, proved  in  \cite{ggp2018} : 

\begin{align}
\Vert S_nf -T_nf \Vert_J \leq  \sum_{k=0}^{n-1} \Vert (S_1-T_1)T_kf \Vert_J \label{eq:Sn-Tn}
\end{align}
Thus we are reduced to the study of $\Vert (S_1-T_1)g \Vert_J$ for $g$ regular enough. We start by the following lemma :

\begin{lemma}
\label{lem:S1-T1}
There exists $\gamma_0, \gamma_1, \gamma_2 > 0$, such that for all $g \in C^3( \mathscr{I})$, 
\[
\|(S_1 - T_1)g\|_J \leq \gamma_0 \|g\|_J + \gamma_1 \frac{ \|g^{(1)}\|_J}{J^2} + \gamma_2 \frac{ \|g^{(2)}\|_J}{ J^2} + 2 \sum_{s ' \neq s} \left| P_{s,s'}^J - \frac{\alpha_s Q_{s,s'}}{J} \right| \|g\|_J + O( \frac{1}{J^3}).
\]
\end{lemma}

\begin{proof}
We first show that for $g$ which is $C^3$,
\[
T_1 g (x,s) = g(x,s) + L g(x,s) + \frac{1}{2}L^2 g(x,s) + O(\frac{1}{J^3})
\]
Let $(\mathscr{P}_t)_{t \geq 0}$ be the semigroup associated with the PDMP $(U_t)_{t \geq 0}$ : $Q_t g(x,s) = \mathbb{E}_{x,s}(g(U_t))$ and $\mathscr{L}$ its generator :
\[
\mathscr{L} g(x,s) = \frac{s x (1-x) }{(1+sx)} \frac{\partial}{\partial x} g(x,s) + \sum_{s' \in E} \alpha_s Q_{s,s'} \left( f(x,s') - f(x,s) \right).
\]
Then, one has $T_t = \mathscr{P}_{t/J}$ and $L = \frac{1}{J} \mathscr{L}$. It is well known that for $g$ in the domain of $\mathscr{L}$,
\[
\frac{\partial }{\partial t} \mathscr{P}_t g(x,s) = \mathscr{P}_t \mathscr{L} g = \mathscr{L} \mathscr{P}_t g
\]
In particular, if $g$ and $\mathscr{L} g$ are in the domain of $\mathscr{L}$, then $\mathscr{P}_t g$ is twice differential with respect to time and 
\[
\frac{\partial^2 }{\partial t^2} \mathscr{P}_t g(x,s) = \mathscr{P}_t \mathscr{L}^2 g = \mathscr{L}^2 \mathscr{P}_t g,
\]
 where $\mathscr{L}^2 g = \mathscr{L}( \mathscr{L}g) $. In the case of PDMP, it is possible to prove that when $g$ is $C^3$, then both $g$ and $\mathscr{L}g$ are in $C^3$. 
 
 Thus, we obtain the following Taylor development :
 \begin{align*}
T_1 g(x,s) & = \mathscr{P}_{t/J} g(x,s)\\
& = g(x,s) + \frac{1}{J}\mathscr{L} g(x,s) +  \frac{1}{2J^2}\mathscr{L^2} g(x,s) + O(\frac{1}{J^3})\\
& =  g(x,s) + L g(x,s) + \frac{1}{2}L^2 g(x,s) + O(\frac{1}{J^3}).
\end{align*}
 Hence, 
 \[
 S_1 g - T_1 g = S_1 g - g - Lg - \frac{1}{2}L^2 g + O(\frac{1}{J^3}).
 \]
 Moreover, we have
\begin{align*}
S_1 g(x,s) & = \mathbb{E}_{x,s} \left( g(X_1^J, s_1^J) \right)\\
& = P_{s,s}^J \mathbb{E}_{x,s} \left( g(X_1^J, s) \right) + \sum_{s' \neq s} P^J_{s,s'} \mathbb{E}_{x,s} \left( g(X^J_1, s') \right),
\end{align*}
which leads to (we drop the $J$ in $P_{s,s'}$ for better readability)

\begin{align*}
S_1 g(x,s) - T_1 g(x,s) & = \mathbb{E}_{x,s} \left( g(X_1^J, s) \right)  - g(x,s) - L_C g(x,s) \\
& \quad (P_{s,s'} - 1) \left( \mathbb{E}_{x,s} (g(X_1^J,s)) - g(x,s)\right) \\
& \quad + \sum_{s' \neq s} P_{s,s'} \left( \mathbb{E}_{x,s} (g(X_1^J,s')) - g(x,s')   \right) \\ 
& \quad + \sum_{s' \neq s}\left( P_{s,s'} - \frac{\alpha_s Q_{s,s'}}{J} \right)\left( g(x,s') - g(x,s) \right) \\
& \quad - \frac{1}{2}L^2 g(x,s) + O(\frac{1}{J^3}).
\end{align*}
We now prove that the three first terms are of order $1/J^2$, with bounds controllable by the derivative of $g$. Note that by definition of $L^2$, it is immediate that there exist some constants $\gamma_0'$, $\gamma_1'$ and $\gamma_2'$ such that
\[
\|L^2g\|_J \leq \gamma_0' \frac{\|g\|_J}{J^2} + \gamma_1' \frac{\|g^{(1)}\|_J}{J^2} + \gamma_2' \frac{\|g^{(2)}\|_J}{J^2} .
\]
Since in one step, the difference $X_1^J - x$ is of order $1/J$, we have the following Taylor development :

\[
\mathbb{E}_{x,s}\left(g( X_1^J, s')\right) =g(x,s') + \frac{\partial g}{\partial x}(x,s') \mathbb{E}_{x,s}[X_1^J - x] + \frac{\partial^2 g}{\partial x^2}(x,s') \frac{1}{2} \mathbb{E}_{x,s}\left[(X_1^J - x)^2\right] + O(\frac{1}{J^3})
\]
 By Proposition \ref{hyp:limitMC} and definition of $L_C$, we thus have 
 \[
\left| \mathbb{E}_{x,s}\left(g( X_1^J, s)\right) - g(x,s) - L_c g(x,s)\right|  \leq \frac{\|g^{(2)}\|}{2J^2} + O(\frac{1}{J^3}).
\]
The previous Taylor development also gives that 
\[
\left| \mathbb{E}_{x,s}\left(g( X_1^J, s')\right) - g(x,s') \right| \leq \frac{\|g^{(1)}\|}{J} + O(\frac{1}{J^2})
\]
Now, since for all $s' \neq s$, $\lim_{J \to \infty} P^J_{s,s'} J = \alpha_s Q_{s,s'}$, we have $P_{s,s'} = O(1/J)$, which concludes the proof.
\end{proof}
Let $f \in C^3( \mathscr{I})$. In view of inequality \ref{eq:Sn-Tn}, we want to apply Lemma \ref{lem:S1-T1} to $T_k f$, for $0 \leq k \leq n-1$. For this, we need to prove that $T_t f$ is regular for all $t \geq 0$, and to give an estimate on the derivative of $T_t f$ :

\begin{lemma}
\label{lem:controlPt}
Let $C = \max_{(x,s) \in \mathscr{I}}| \frac{\partial G_s}{\partial x}(x)|$ and $K = \max_{(x,s) \in \mathscr{I}}| D^2G_{s}(x)|$. Then  for all $f \in C^2(\mathscr{I})$, for all $t \geq 0$,
\[
\|( \mathscr{P}_t f)^{(1)}\| \leq \exp(Ct)  \|f^{(1)}\|,
\]
and
\[
\|( \mathscr{P}_t f)^{(2)}\| \leq \exp(2Ct)  \left( \|f^{(1)}\| + K \|f^{(2)}\| \right).
\]
In particular, 
\[
\|( T_t f)^{(1)}\| \leq \exp(C\frac{t}{J})  \|f^{(1)}\|,
\]
and
\[
\|( T_t f)^{(2)}\| \leq \exp(2C\frac{t}{J})  \left( \|f^{(1)}\| + K \|f^{(2)}\| \right).
\]
\end{lemma}
\begin{proof}
For $s \in E$, we denote by $\varphi_t^s(x)$ the flow generated by $G_s$ and started at $x$. That is, $\varphi_t^s(x)$ is solution to
\[
\begin{cases}
\frac{\partial \varphi_t^s(x)}{\partial t} = G_s(\varphi_t^s(x))\\
\varphi_0^s(x) = x.
\end{cases}
\]
By $C^{\infty}$ regularity of $G_s$, for every $t \geq 0$, $x \mapsto \varphi_t^s(x)$ is $C^{\infty}$.
Let $(\tilde{S}_n)_{n \geq 0}$ and $(\tilde{T}_n)_{n \geq 0}$ denote the sequence of postjump locations and of jump times of $(s_t)_{t \geq 0}$, respectively. Under $\mathbb{P}_{x,s}$, one has
\[
X_t^x = \varphi_{t - \tilde{T}_{N_t}}^{\tilde{S}_{N_t}} \circ \ldots \circ \varphi_{\tilde{T}_1}^s(x),
\]
where we have denoted $X_t^x$ the process $X_t$ to emphasis the dependence on $x$. From this equation, since the sequences $(\tilde{S}_n)_{n \geq 0}$ and $(\tilde{T}_n)_{n \geq 0}$ do not depend on $x$ and since every $\varphi^s$ is $C^{\infty}$, one deduce that $x \mapsto X_t^x$ is $C^{\infty}$ almost surely. We can also write
\[
X_t^x = x + \int_0^t G_{s_u}(X_u^x)du,
\]
that we can differentiate  with respect to $x$ to find 
\[
\frac{\partial}{\partial x} X_t^x = 1 + \int_0^t DG_{s_u}(X_u^x) \frac{\partial}{\partial x} X_u^x du.
\]
Using Gronwall's inequality, we deduce that 
\[
\left| \frac{\partial}{\partial x} X_t^x \right|  \leq e^{Ct}.
\]
Differentiating once again, we find 
\[
\frac{\partial^2}{\partial x^2} X_t^x = \int_0^t D^2G_{s_u}(X_u^x) \left( \frac{\partial}{\partial x} X_u^x \right)^2 du + \int_0^t DG_{s_u}(X_u^x) \frac{\partial^2}{\partial x^2} X_u^x du,
\]
which leads to
\[
\left| \frac{\partial^2}{\partial x^2} X_t^x \right| \leq K (e^{Ct} - 1) + C \int_0^t \left| \frac{\partial^2}{\partial x^2} X_u^x \right| du.
\]
Using once again Gronwall's inequality, we deduce that 
\[
\left| \frac{\partial^2}{\partial x^2} X_t^x \right| \leq K (e^{Ct} - 1) e^{Ct} \leq K e^{2Ct}.
\]
Now let $f \in C^2(\mathscr{I})$, then
\[
\frac{\partial}{\partial x} f(X_t^x, s_t) = \left( \frac{\partial}{\partial x} f \right)( X_t^x, s_t)\frac{\partial}{\partial x} X_t^x, 
\]
and thus 
\[ 
\left| \frac{\partial}{\partial x} f(X_t^x, s_t) \right| \leq \|f^{(1)} \| e^{Ct}.
\]
Furthermore, 
\[
\frac{\partial^2}{\partial x^2} f(X_t^x, s_t) = \left(\frac{\partial^2}{\partial x^2}f\right)(X_t^x,s_t) \left( \frac{\partial}{\partial x} X_t^x \right)^2 + \left( \frac{\partial}{\partial x} f \right)( X_t^x, s_t)\frac{\partial^2}{\partial x^2} X_t^x, 
\]
hence
\[
\left| \frac{\partial^2}{\partial x^2} f(X_t^x, s_t) \right| \leq e^{2Ct} \left( \|f^{(2)} \| + K \|f^{(1)} \| \right) 
\]

\end{proof}
We can now finish the proof of Theorem \ref{thqttapprox}. Using inequality \eqref{eq:Sn-Tn} and Lemma \ref{lem:S1-T1} applied to $g = T_k f$, we have :

\begin{align*}
    \|S_n f - T_n f \| &  \leq  \sum_{k=0}^{n-1} \|(S_1 - T_1) T_kf\|\\
    & \leq \frac{1}{J^2}\sum_{k=0}^{n-1}\left( \gamma_0 \|T_kf\| + \gamma_1 \|T_k^{(1)}f\| + \gamma_2 \|T_k^{(2)}f\| \right)\\
    & \quad +  \sum_{k=0}^{n-1}\left[ 2 \sum_{s' \neq s} \left| P_{s,s'}^J - \frac{\alpha_sQ_{s,s'}}{J}\right|\|T_kf\| + O(\frac{1}{J^3}) \right].
\end{align*}
Now, since $n = \lfloor tJ \rfloor$, we get that $ \sum_{k=0}^{n-1} O(\frac{1}{J^3}) = O(\frac{1}{J^2})$. Moreover, by Lemma \ref{lem:controlPt}, we have
\[
\sum_{k=0}^{n-1} \|(T_kf)^{(1)}\| \leq  \|f^{(1)}\| \sum_{k=0}^{n-1} e^{k \frac{C}{J}}=\|f^{(1)}\| \frac{e^{n\frac{C}{J}} - 1}{e^{\frac{C}{J}}-1}
\]
Since $e^{\frac{C}{J}}-1 \geq \frac{C}{J}$ and $n = \lfloor tJ \rfloor \leq t J$, we get 
\[
\sum_{k=0}^{n-1} \|(T_kf)^{(1)}\| \leq  \|f^{(1)}\|J \frac{e^{tC}-1}{C}.
\]
For the same reason, we get that
\[
\sum_{k=0}^{n-1} \|(T_kf)^{(2)}\| \leq  \left(\|f^{(1)}\| + K \|f^{(2)}\|\right)J \frac{e^{2tC}-1}{2C},
\]
hence
\begin{align*}
    \frac{1}{J^2}\sum_{k=0}^{n-1}( \gamma_0 \|T_kf\| + & \gamma_1 \|T_k^{(1)}f\| + \gamma_2 \|T_k^{(2)}f\|)   \\
    & \leq \frac{1}{J}\left( t \gamma_0 \|f\| + \gamma_1 \frac{e^{t2C}-1}{C}\|f^{(1)}\| + \gamma_2 K \frac{e^{2tC}-1}{2C} \|f^{(2)}\|\right)
\end{align*}
This, together with the fact that
\[
\sum_{k=0}^{n-1}  \sum_{s' \neq s} \left| P_{s,s'}^J - \frac{\alpha_sQ_{s,s'}}{J}\right|\|T_kf\| \leq \left| J P_{s,s'}^J - \alpha_sQ_{s,s'}\right|\|f\|
\]
conclude the proof of Theorem \ref{thqttapprox}.

\subsection{A lemma for accessibility}
This section is devoted to the statement and the proof of a lemma which is useful for accessibility issue. Let $E$ be a finite set, and for all $i \in E$, let $f_i : \mathbb{R}^d \to \mathbb{R}^d $ be a globally integrable vector field. For $i \in E$, we denote by $(\varphi_t^i(x))_{t \geq 0}$ the flow generated by $f_i$ and started at $x$. That is, $\varphi_t^i(x)$ is solution to
\[
\begin{cases}
\frac{\partial \varphi_t^i(x)}{\partial t} = f_i(\varphi_t^s(x))\\
\varphi_0^i(x) = x.
\end{cases}
\]
 For $m \geq 1$,  $\mathbf{i}=(i_1, \ldots, i_m) \in \{1,2,3\}^m$ and $\mathbf{u} = (u_1,\ldots,u_m) \in \mathbb{R}_+^m$, we denote by $\mathbf{\Phi}_{\mathbf{u}}^{\mathbf{i}}$ the composite flow: $\mathbf{\Phi}_{\mathbf{u}}^{\mathbf{i}} = \varphi_{u_m}^{i_m} \circ \ldots \circ \varphi_{u_1}^{i_1}$.  For $x \in \mathbb{R}^d$ and $t \geq 0$, we denote by $\mathcal{O}^+_{f,t}(x)$ (resp. $\mathcal{O}^+_f(x)$)  the set of points that are reachable from $x$ at time $t$ (resp. at any nonnegative time) with a composite flow: $$\mathcal{O}^+_{f,t}(x)=\{ \mathbf{\Phi}_{\mathbf{v}}^{\mathbf{i}} (x), \: (\mathbf{i},\mathbf{v}) \in E^m \times \mathbb{R}_+^m, m \in \mathbb{N}, v_1 + \ldots + v_m = t\},$$

$$ 
\mathcal{O}^+_f(x) = \bigcup_{t \geq 0} \mathcal{O}^+_{f,t}(x).
$$ 
For $B \subset \mathbb{R}^d$, We define the set of points that are accessible from $B$ (with the vector fields $(f_i)_{i \in E}$) as 
\[
\Gamma_{f,B} = \bigcap_{ x \in  B} \overline{\mathcal{O}_f^+(x)}.
\]
These accessible sets are linked to the support of invariant probabiliy measure of PMDP, see e.g. \cite[Proposition 3.17]{BMZIHP}. With these notations, we have the following lemma :

\begin{lemma}
\label{lem:accessconjug}
Let $(f_i)_{i \in E}$ and $(g_i)_{i \in E}$ be two families of globally integrable vector fields on $\mathbb{R}^d$. Assume that for all $i$, there exists a positive, Lipschitzian,  bounded function $h_i : \mathbb{R}^d \to \mathbb{R}$ such that $g_i = h_i f_i$. Then, for all $x \in \mathbb{R}^d$, one has 
\[ 
\mathcal{O}^+_f(x) = \mathcal{O}^+_g(x).
\]
In particular, for all $B \subset \mathbb{R}^d$
\[
\Gamma_{f,B} = \Gamma_{g,B}.
\]
\end{lemma}

\begin{proof}
Since  $g_i = h_i f_i$, one can find, for all $i \in E$ and $x \in \mathbb{R}^d$ an increasing bijection $\alpha^i_x : [0, + \infty) \to [0 + \infty)$ such that, for all $t \geq 0$, $\psi_t^i(x) = \varphi^i_{\alpha^i_x(t)}(x)$, where $\psi^i$ and $\varphi_i$ are the flow generated by $g_i$ and $f_i$, respectively. Now, let $y \in \mathcal{O}^+_f(x)$. Then, there exists $(\mathbf{i},\mathbf{v}) \in E^m \times \mathbb{R}_+^m$ such that 
\[
y = \mathbf{\Phi}_{\mathbf{u}}^{\mathbf{i}}(x) = \varphi_{u_m}^{i_m} \circ \ldots \circ \varphi_{u_1}^{i_1}(x).
\]
Hence, 
\[
y = \psi_{\alpha^{i_m}_{\varphi_{u_{m-1}}^{i_{m-1}} \circ \ldots \circ \varphi_{u_1}^{i_1}(x)}(u_m)}^{i_m} \circ \ldots \circ \psi^{i_1}_{\alpha^{i_1}_x(u_1)}(x),
\]
and so $y \in \mathcal{O}^+_g(x)$. Since the functions $\alpha$ are bijective, we can prove the converse inclusion, hence the equality of $ \mathcal{O}^+_f(x)$ and $\mathcal{O}^+_g(x)$
\end{proof}
With the notations of the main sections of this paper, one has 
\[
G_s(x) = h_s(x) f_s(x),
\]
where $h_s(x) = 1 + s^1 x^1 + \ldots + s^S x^S$ and $f_s^i(x) = x^i(s^i - (s^1 x^1 + \ldots + s^S x^S))$. The particular interest of the lemma lies in the fact that a convex combination of the $f_s$ is given by the function with the convex combination of the coefficients. That is, if $(a_s)_{s \in E}$ are nonnegative such that $\sum_s a_s = 1$, and if we set for $s_a = \sum_{s \in E} a_s s$, then $\sum_s a_s f_{s} = f_{s_a}$.

\section*{Acknowledgments}
We thank Michel Benaïm for many discussions on this subject, and for making it possible for AP and ES to meet. This work was mainly done while ES was at Université de Neuchâtel, and is partially founded by the SNF grant 200021\_175728. \\
We also thank Franck Jabot for having suggested the study of Moran's model and his varied ideas which allowed us to give biological meaning to this article.
\bibliographystyle{plain}
\bibliography{biblio}

\end{document}